\documentclass[centertags,oneside,
12pt]{amsart}

\usepackage{amstext    }
\usepackage{amsthm    }
\usepackage{a4}
\usepackage[mathscr]{eucal}
\usepackage{mathrsfs}
\usepackage{hyperref}
\usepackage{charter}
\usepackage{amsmath}
\usepackage{amssymb}
\usepackage{amscd}

\numberwithin{equation}{section}

\usepackage[a4paper,width=16.2cm,top=3cm,bottom=3cm]{geometry}

\newtheorem{theorem}{Theorem}[section]
\newtheorem{definition}[theorem]{Definition}
\newtheorem{proposition}[theorem]{Proposition}
\newtheorem{corollary}[theorem]{Corollary}
\newtheorem{lemma}[theorem]{Lemma}
\newtheorem{remark}[theorem]{Remark}

\newcommand{\cali}[1]{\mathscr{#1}}

\newcommand{\Tan}{\mathop{\mathrm{Tan}}\nolimits}
\newcommand{\Cotan}{\mathop{\mathrm{Cotan}}\nolimits}
\newcommand{\Nor}{\mathop{\mathrm{Nor}}\nolimits}

\newcommand{\Leb}{{\rm Leb}}

\newcommand{\Vol}{{\rm Vol}}

\newcommand{\dist}{{\rm dist}}

\newcommand{\chac}{{\rm \bf 1}}

\newcommand{\loc}{{loc}}
\newcommand{\ddc}{{dd^c}}
\newcommand{\ddczw}{{dd^c_{z,w}}}
\newcommand{\ddcy}{{dd^c_{y}}}
\newcommand{\ddcz}{{dd^c_{z}}}
\newcommand{\dc}{{d^c}}
\newcommand{\dbar}{{\overline\partial}}
\newcommand{\ddbar}{{\partial\overline\partial}}

\newcommand{\ind}{{\bf 1}}
\newcommand{\id}{{\rm id}}

\newcommand{\hol}{{\rm hol}}

\newcommand{\lof}{\mathop{\mathrm{{log^\star}}}\nolimits}

\renewcommand{\Re}{{\rm Re}}
\renewcommand{\Im}{{\rm Im}}

\newcommand{\Ac}{\cali{A}}

\newcommand{\Cc}{\cali{C}}
\newcommand{\Dc}{\cali{D}}

\newcommand{\Fc}{\cali{F}}

\newcommand{\Lc}{\cali{L}}

\newcommand{\Oc}{\cali{O}}

\newcommand{\Qc}{\cali{Q}}

\newcommand{\Uc}{\cali{U}}
\newcommand{\Vc}{\cali{V}}

\newcommand{\Ic}{\cali{I}}

\newcommand{\A}{\mathbb{A}}
\newcommand{\B}{\mathbb{B}}
\newcommand{\C}{\mathbb{C}}
\newcommand{\D}{\mathbb{D}}
\newcommand{\E}{\mathbb{E}}
\renewcommand{\H}{\mathbb{H}}
\newcommand{\N}{\mathbb{N}}
\newcommand{\Z}{\mathbb{Z}}
\newcommand{\R}{\mathbb{R}}
\newcommand{\T}{\mathbb{T}}
\newcommand{\U}{\mathbb{U}}
\newcommand{\V}{\mathbb{V}}
\renewcommand{\S}{\mathbb{S}}
\renewcommand{\P}{\mathbb{P}}

\title[Negative Lyapunov exponent]{Singular holomorphic  foliations by curves  II: Negative Lyapunov exponent}

\author{ Vi{\^e}t-Anh Nguy{\^e}n}

\date{May 29, 2020}

\begin{document}

\maketitle

\begin{abstract}
Let $\Fc$ be  a  holomorphic  foliation by Riemann surfaces defined on a   compact complex projective surface $X$ satisfying the following  two conditions:

$\bullet$  the   singular points of $\Fc$ are all hyperbolic;

 $\bullet$ $\Fc$ is  Brody hyperbolic.
 
 Then we establish  cohomological formulas for the Lyapunov  exponent and  the Poincar\'e mass   of 
 an extremal positive  $\ddc$-closed  current tangent to $\Fc.$

 If, moreover, there is no   nonzero positive closed  current  tangent to $\Fc,$ 
then 
we  show that the   Lyapunov exponent $\chi(\Fc)$ of $\Fc,$  which is, by definition,  the   Lyapunov exponent
of the unique  normalized  positive  $\ddc$-closed  current tangent to $\Fc,$
 is a strictly negative   real number.
 
 As  an application, we compute the Lyapunov  exponent of a generic  foliation with a given degree 
 in $\P^2.$
  \end{abstract}

\bigskip

\bigskip

\noindent
{\bf Classification AMS 2010:} Primary: 37F75, 37A30;  Secondary:  57R30, 58J35, 58J65, 60J65.




 \section{Introduction} \label{intro}


Let $\Fc=(X,\Lc,E)$    be  
a   holomorphic foliation  by curves in a    compact K\"ahler  surface $X$  with      the set of  singularities $E.$ 
Recall that the foliation $\Fc$ is given by  an open covering $\{\U_j\}$ of $X$ and  holomorphic vector fields $v_j\in H^0(\U_j,\Tan(X))$ with isolated singularities (i.e. isolated zeroes) such that
\begin{equation}\label{e:Cotan}
v_j=g_{jk}v_k\qquad\text{on}\qquad \U_j\cap \U_k
\end{equation}
for some non-vanishing holomorphic functions $g_{jk}\in H^0(\U_j\cap \U_k, \Oc^*_X).$
Its leaves are locally integral curves of these vector fields.  The set of singularities $E$ of $\Fc$ is  precisely the union of the zero  sets of these local vector fields. The set $E$ is finite.
 We say that a
singular point $a\in E $ is  {\it linearizable}  if 
  there are  local holomorphic coordinates $(z,w)\in\C^2$ centered at  $a$  such that the
leaves of $\Fc$ are integral curves of a  vector field 
 $$Z(z,w) = z {\partial\over \partial z}
+ \lambda w{\partial\over \partial w}\qquad\text{with some complex number}\qquad \lambda\not =0.$$ 
The   analytic curves $\{  z=0\} $ and   $\{  w=0\}$ are  called  {\it separatrice}  at $a.$
If moreover, $\lambda\not\in\R,$ then  we say that $a$ is {\it hyperbolic}.
 
The  functions $g_{jk}$ form a multiplicative cocycle and hence give a  cohomology class in $H^1(X,\Oc^*_X),$ that is,  a holomorphic line  bundle on $X.$
It is  called the {\it cotangent bundle } of $\Fc,$ and is denoted by $\Cotan(\Fc).$ Its  dual $\Tan(\Fc),$ represented by the inverse  cocycle $\{g_{jk}^{-1}\}$, is called the {\it tangent  bundle } of $\Fc.$    

Recall also that
 a positive $\ddc$-closed current $T$ of bidegree $(1, 1)$ on $X$ is {\it directed by}  the foliation $\Fc$ (or equivalently, {\it tangent to} $\Fc$) 
if $T \wedge \Phi = 0$ for every local holomorphic $1$-form $\Phi$ defining $\Fc .$
The directed $\ddc$-closed  currents are  generalizations of the  {\it foliations cycles}
introduced by  Sullivan  \cite{Sullivan}.

Two fundamental concepts  associated to  $\Fc$ are its  holonomy cocycle  $\mathcal H$  and the convex cone of  {\it positive $\ddc$-closed   currents} directed by  $\Fc.$ 
 They are 
geometric objects.  An important characteristic  number  which relates these two concepts  is
the Lyapunov exponent $\chi(T)$ of an extremal positive $\ddc$-closed  current $T$ directed by $\Fc.$
In the  first article of this  series,  
we have  established  an effective   sufficient condition for the  existence  of the Lyapunov exponents of  Brody hyperbolic  foliations.
This class  of foliations  was first introduced in our joint-work with Dinh and Sibony   \cite{DinhNguyenSibony14b}.  
Now  we continue, in the  second  article of this  series,
the study  of the Lyapunov exponents.  More   specifically,  we investigate    the interplay between  the  dynamical and geometric  interpretations  of
 these characteristic numbers, and we determine whether these numbers are  positive/zero/negative. 
 The presence  of  singular points makes our analysis delicate.

Let $\Omega=\Omega(\Fc)$ be the sample-path space consisting of all continuous paths $\omega:\R\to X\setminus E$ with image fully contained in a single leaf. 
Let $g_P$ be the leafwise  Poincar\'e metric on $\Fc.$ Consider   the  corresponding  harmonic positive measure
 \begin{equation}\label{e:mu} \mu:= T\wedge  g_P\quad\text{ on}\quad X\setminus E,\quad\text{and}\quad \mu(E):=0.
 \end{equation}
 The {\it Poincar\'e mass} of $T$ is, by definition,  the mass $\|\mu\|:=\int_X d\mu$ of the measure $\mu.$  
When all points  $a\in E$ are linearizable,  by  \cite[Proposition 4.2]{DinhNguyenSibony12} $\|\mu\|$ is  finite.
For $x\in X\setminus E,$  the restriction of $g_P$   on the leaf $L_x$ passing through $x,$ generates the corresponding leafwise Brownian motion.
This Markov process defines a canonical probability measure, the Wiener measure $W_x$  on $\Omega$ which gives full mass to  the subspace $\Omega_x$
consisting of all paths $\omega\in\Omega$ with $\omega(0)=x$  (see Subsection \ref{SS:Wiener} below).  
Now  we  recall the following  existence theorem for the Lyapunov exponents.

\begin{theorem}\label{T:VA}  {\rm  (\cite[Theorem 1.1]{NguyenVietAnh18b}).}
Let $\Fc=(X,\Lc,E)$ be  a  holomorphic  foliation  by Riemann surfaces defined on a    Hermitian compact complex projective surface $X$ satisfying the following  two conditions:

$\bullet$  its singularities  $E$ are all hyperbolic;

 $\bullet$ $\Fc$ is  Brody hyperbolic.
 
 Let $T$ be  a positive $\ddc$-closed   current  directed by $\Fc$  which does not give mass to any  invariant  analytic curve.
Assume, in addition, that $T$ is an  extremal element in the convex cone of  all  positive $\ddc$-closed   currents  directed by $\Fc.$

   Then    \begin{enumerate}
   \item $T$ admits
the (unique) {\bf Lyapunov exponent} $\chi(T)$  given by the  formula
\begin{equation}\label{e:Lyapunov_exp}
\chi(T):= \int_X\big(\int_{\Omega} \log \|\mathcal{H}(\omega,1)\| dW_x(\omega)\big)d\mu(x).  
\end{equation}

\item
  For  $\mu$-almost  every $x\in X\setminus E,$  we have
 $$
 \lim\limits_{t\to \infty} {1\over  t} \log  \| \mathcal H(\omega,t) \|=\chi(T)  
 $$
  for    almost every path  $\omega\in\Omega$ with respect to $W_x.$
  \end{enumerate}
\end{theorem}
 In fact, assertion (1) is a  consequence of  the so-called  {\it integrability of the  holonomy cocycle.}
Assertion (2) says that  
the  characteristic  number $\chi(T)$    measures  heuristically   the exponential  rate of  convergence  of leaves toward each other  along  leafwise Brownian trajectories
 (see Candel  \cite{Candel03},  Deroin \cite{Deroin05} for the nonsingular case).   
 Therefore, Theorem \ref{T:VA} gives a  dynamical  characterization of $\chi(T).$

 Now  we  discuss  the geometric  aspect of $\chi(T).$ Let $H^{1,1}(X)$ denote the Dolbeault cohomology
group of real smooth $(1,1)$-forms on $X.$ For  a real smooth closed $(1,1)$-form $\alpha$ on $X,$ let $\{\alpha\}$ be  its  class
in $H^{1,1}(X).$
The cup-product $\smile$ on $H^{1,1} (X) \times H^{1,1} (X)$ is defined by
$$
(\{\alpha\},\{\beta\}) \mapsto \{\alpha\} \smile \{\beta \} :=\int_X\alpha\wedge\beta,
$$
where $\alpha$ and $\beta$ are real smooth closed forms. The last integral depends only on
the classes of $\alpha$  and $\beta.$ The bilinear form $\smile$ is non-degenerate and induces
a canonical  isomorphism  between  $H^{1,1} (X)$ and its  dual  $H^{1,1} (X)^*$  (Poincar\'e duality). In
the definition of $\smile$ one can take $\beta$ smooth and $\alpha$ a current in the sense of
de Rham. So, $H^{1,1} (X)$ can be defined as the quotient of the space of real
closed $(1, 1)$-currents by the subspace of $d$-exact currents. Recall that an $(1, 1)$-current $\alpha$ is real  (resp. $\ddc$-closed)  if $\alpha = \bar \alpha$ (resp. $\ddc\alpha=0$).
Assume that $\alpha$ is a real
$\ddc$-closed $(1, 1)$-current. 
 Then 
by the $\ddc$-lemma, the integral $\int_X \alpha \wedge \beta$ is also independent of the
choice of $\beta$ smooth and closed in a fixed cohomology class. So, using the above ismomorphism,
one can associate to  such $\alpha$ a class $\{\alpha\}$ in $H^{1,1} (X).$
 For a complex  line  bundle $\E$ over $X,$ let  $c_1(\E)$ denote the cohomology Chern class of $\E.$
 This is an element  in $H^{1,1}(X).$

 
 Our first  main  result gives cohomological formulas for $\chi(T)$ and $\|\mu\|$ in terms of the geometric quantity $T$ and some   characteristic classes
 of $\Fc.$
\begin{theorem}\label{T:Main_1}  {\rm (Theorem A) }
Under the  assumption of Theorem  \ref{T:VA},
the following  identities hold
   \begin{subequations}\label{e:thm_A}
     \begin{alignat}{1}
\chi(T)&=  - c_1(\Nor(\Fc))\smile\{T\},\label{e:coho_exponent}\\ 
 \|\mu\|&= c_1(\Cotan(\Fc))\smile\{T\}\label{e:coho_mass}.
  \end{alignat}
\end{subequations}
Here  $\Nor(\Fc):=\Tan(X)/\Tan(\Fc)$ 
stands for the normal bundle 
of $\Fc,$  where $\Tan(X)$ (resp. $\Tan(\Fc)$ and $\Cotan(\Fc)$) is  as  usual the  tangent bundle of $X$ (resp. the tangent bundle  and the  cotangent bundle of $\Fc$).
\end{theorem}

On the  other hand, in collaboration with   Dinh and  Sibony, we have recently proved that in  most interesting cases,
the directed positive $\ddc$-closed   current exists uniquely   (up to  a  multiplicative constant).

\begin{theorem}
 \label{T:DNS} {\rm (\cite[Theorem 1.1]{DinhNguyenSibony18}) } 
 Let $\Fc$    be  
a    holomorphic foliation by   Riemann surfaces  with   only hyperbolic singularities  in a    compact K\"ahler surface $X.$ 
Assume that $\Fc$ admits no directed positive closed current.
Then there exists a  unique  positive  $\ddc$-closed   current $T$ of Poincar\'e mass  $1$  directed  by $\Fc.$  
\end{theorem}
Theorem \ref{T:DNS} implies strong ergodic properties  for  the  foliation  $\Fc.$ It is  worthy noting that
when $X=\P^2$  the theorem was obtained by  Forn{\ae}ss--Sibony  \cite{FornaessSibony10}, see also  Dinh--Sibony \cite{DinhSibony18a} and P\'erez-Garrand\'es \cite{Perez-Garrandes}.


 Suppose now  that $\Fc$    is  
a    holomorphic foliation by   Riemann surfaces  with   only hyperbolic singularities  in a    compact projective surface $X$ 
such  that $\Fc$ admits no directed positive closed current.
So the assumptions of both Theorem  \ref{T:VA} and \ref{T:DNS}  are fulfilled.
 By  Theorem \ref{T:DNS}, let $T$ be  the unique  directed positive  $\ddc$-closed   current $T$ whose the Poincar\'e mass  is  equal to $1.$ 
 \begin{definition}\label{D:Lyapunov}  \rm The {\it Lyapunov  exponent}  of  the foliation $\Fc,$  denoted by $\chi(\Fc),$  is by  definition, the  real number  $\chi(T)$ given by  Theorem \ref{T:VA}.
 \end{definition}
 
 When  we  explore the dynamical system  associated to a foliation $\Fc,$  the sign of its Lyapunov   exponent is a crucial   information.
 Indeed, the positivity/negativity of $\chi(\Fc)$ corresponds to the   repelling/attracting   character of a typical leaf  along a typical Brownian trajectory.    
 Here is  our   second main result. 

\begin{theorem}\label{T:Main_2}  {\rm (Theorem B) }
 Under the  assumption of Theorem  \ref{T:DNS}, assume in addition that $X$ is projective.
Then  $\chi(\Fc)$ is a negative  nonzero real number.
\end{theorem}

Roughly speaking,  Theorem B says that in the sense of ergodic theory, generic leaves have  the  tendancy to wrap together towards the support of the unique normalized 
directed positive $\ddc$-closed current.

 Now  we apply   the above results to  the family of  singular holomorphic  foliations  on $\P^2$ with  a given degree $d>1.$ 
 Recall that the degree is the number of tangencies of the foliation with a generic line.
 This  family can be identified  with  a Zariski dense open set $\Uc_d$ of some projective space. 
By Brunella  \cite{Brunella06}, 
if $\Fc\in\Uc_d$ with the  properties that all the singularities of $\Fc$
 are hyperbolic and that $\Fc$ does not possess any invariant algebraic curve, then $\Fc$ admits no nontrivial
directed positive closed current, and hence $\Fc$  satisfies
the assumptions of both  Theorems A and B.
By  Jouanolou \cite{Jouanolou} and  Lins Neto-Soares \cite{NetoSoares}, these properties
   are  satisfied
for  $\Fc$ in  a  set of full Lebesgue measure
 of $\Uc_d.$
Consequently, Theorems A and B apply and   give us the following result.
It can be   applied  to  every generic  foliation in $\P^2$ with a given degree $d>1.$
  
\begin{corollary}\label{C:Thm_A}
 Let $\Fc=(\P^2,\Lc,E)$    be  
a   singular  foliation by curves on the complex projective plane $\P^2.$ Assume that
all the singularities  are hyperbolic and  that  $\Fc$ has no invariant algebraic  curve. 
 Then  
 \begin{equation}\label{e:univ_exponent}\chi(\Fc)=-{d+2\over d-1},
 \end{equation}
 where $d$ is the degree of $\Fc.$
\end{corollary}  
To the best of our knowledge, this is  the first  family  of singular holomorphic foliations for which one knows to compute the Lyapunov exponents.

\medskip

Now  we discuss  the relations between our results and   previous works.
Candel \cite{Candel93}  introduces  the
Euler class of a  directed positive harmonic current in the context of   compact  Riemann surface laminations  endowed with a conformal metric.
So  his notion generalizes the intersection of
the Chern class of a holomorphic line bundle with a  directed positive harmonic current
considered  in Theorem A,  but only  when the foliation is  nonsingular. Moreover, in the context of general compact laminations,  Candel  introduces in \cite{Candel03}
the notion of   one-dimensional cocycle. This  notion   has  the  advantage of not using  any structure of complex line bundles.
He obtains an integral formula for   the Lyapunov exponent of such a cocycle
 with respect to a  directed  positive harmonic current. 
This result has been generalized  for higher  dimensional cocycles in \cite{NguyenVietAnh17}.
On the other hand,
a version of Theorem A  for  exceptional minimal sets has been established by Deroin \cite[Appendix A]{Deroin05}. 
Concerning  Theorem B,   the negativity of  $\chi(\Fc)$ has been proved by  Deroin and Kleptsyn \cite[Theorem B]{DeroinKleptsyn}
in the context of transversally  conformal foliations (without singularities), see also Baxendale \cite{Baxendale} for a related result.
Formula \eqref{e:univ_exponent} has already been found out by  Deroin and Kleptsyn \cite[Proposition 3.12]{DeroinKleptsyn} but under  the strong hypothetical additional assumption  that  there is
an   exceptional minimal set.

\medskip

One of  the main ingredients in our proof of Theorem A is  some  precise and delicate  estimates on the variations of the holonomy  cocycle and  the clustering  mass of a 
directed positive $\ddc$-closed current near the singularities. In particular, Proposition \ref{P:Lyapunov_vs_kappa} below plays a decisive role in our approach, it gives a  stronger  
estimate than those obtained previously   in  \cite{DinhNguyenSibony18,FornaessSibony10,NguyenVietAnh18b} etc.
 The proof of  Proposition \ref{P:Lyapunov_vs_kappa} is  partly based on  our last  work \cite{NguyenVietAnh18b}. However, it also requires  new techniques which rely  on estimates on heat diffusions associated to the leafwise Poincar\'e metric $g_P.$

If  we move away all the dynamical aspects of Theorem A, then  its proof   boils down to some continuity  problems of the  wedge-product of a positive $\ddc$-closed $(1,1)$-current with
another (not necessarily positive)  $(1,1)$-current  whose  potential  is  unbounded.
This type of problems  is quite new     in the intersection theory  of currents since up to now only the case with  bounded potential  is  understood (see e.g. \cite{DinhSibony04}).
In this  vein, 
other ingredients in our proof of Theorem A are some tools  from   complex  geometry  such as  the  regularization of  curvature currents of singular Hermitian holomorphic line bundle 
(see \cite{BlockiKolodziej}), 
the positivity of the cotangent  bundle  with respect to the leafwise Poincar\'e metric (see \cite{Brunella03}).

To prove the  negativity of the Lyapunov exponent (i.e. Theorem B), we use  Hahn-Banach separation theorem  following an idea of Sullivan \cite{Sullivan}  and  Ghys \cite{Ghys88}. However, in order to 
carry out this plan  in the context of singular foliations,  we introduce
 an adapted normed space whose norm  is taken with respect to a natural weight function $W.$ This  function reflects  the  singularities of the considered foliation. We make  a full use of 
the key fact  obtained in  Proposition \ref{P:Lyapunov_vs_kappa}  that $W$ is $\mu$-integrable. A systematic  study   on the variations of the holonomy  cocycle near the singularities
of the foliation
is also needed.

Understanding the dynamics of singular holomorphic  foliations  is  a  challenging big  program.
We hope that some of the techniques developed  in this paper may be extended  to singular holomorphic  foliations  in  higher dimensions.
 The reader is  invited  to  consult  P\u{a}un-Sibony \cite{PaunSibony} for a
  fruitful  discussion on the link between value distribution theory and     positive   currents directed by singular holomorphic foliations.
  
\medskip

The  paper is  organized as follows. In  Section   \ref{S:background}, we   recall
some basic  elements of singular holomorphic foliations.
Section \ref{S:preliminaries} is  devoted to a geometric study of the transversal metric and  the holonomy variations near  singularities, which leads to
a natural  weight function $W.$
Based on   this  study  Section \ref{S:Lyapunov} develops a  dynamical study of  the holonomy variations as  well as some
estimates on the  clustering mass of a directed positive $\ddc$-closed  current near  singularities.
Two proofs of the first half of Theorem A (i.e.  identity \eqref{e:coho_exponent}) are given at the end of the  section.
The  last half of  Theorem  A  (i.e.  identity \eqref{e:coho_mass}) is proved in Section \ref{S:Mass}. 
   Section \ref{S:Negativity} is  devoted to the proof of Theorem  B. Finally, we conclude the article with some open questions and remarks.

 After  I had finished  the first version of this  article  in November 2018, Deroin informed me that independently with Kleptsyn, they had obtained a similar result
 but possibly under a stronger hypothesis on $\Fc.$  A  crucial ingredient of their  approach is  my  Theorem \ref{T:VA}. 
 
\smallskip

\smallskip

\noindent
{\bf Acknowledgments. }  
The paper was partially prepared 
during my visit  at the  Vietnam  Institute for Advanced Study in Mathematics (VIASM). 
  I would like to express my gratitude to this organization for hospitality and  for  financial support.

\noindent {\bf Notation.} 
Recall that   $d,$ $\dc$ denote the  real  differential operators on $X$  defined by
$d:=\partial+\overline\partial,$  $\dc:= {1\over 2\pi i}(\partial -\overline\partial) $ so that 
$\ddc={i\over \pi} \partial\overline\partial.$ Throughout the  article, we denote by $\D$   the unit disc in $\C.$
For $r>0$  we denote by $r\D$  the disc in $\C$ with center  $0$  and
with radius $r.$  
  The letters $c,$ $c',$ $c_0,$  $c_1,$ $c_2$ etc. denote  positive constants, not necessarily the same at each  occurrence.  The notation $\gtrsim$ and $\lesssim$ means inequalities  up to  a  multiplicative constant, whereas  we  write  $\approx$ when  both inequalities  are satisfied.
 Let $O$ and $o$ denote   the usual Landau asymptotic  notations. 
Let $\lof(\cdot):=1+|\log(\cdot)|$ be a log-type function. 

 \section{Background}\label{S:background}
 

Let $X$ be  a compact complex surface  endowed with a smooth Hermitian metric $g_X.$
Let $\Fc=(X,\Lc,E)$ be a  holomorphic foliation by Riemann surfaces, where the   set of  singularities $E$ is  finite.
For a recent account on singular holomorphic foliations, the reader is invited to
consult the survey articles  \cite{DinhSibony18b,FornaessSibony08, NguyenVietAnh18c, NguyenVietAnh19}.

\subsection{Poincar\'e metric and Brody hyperbolicity}

Let $g_P$ be  the Poincar\'e metric 
on the unit disc  $\D,$  defined  by
$$ g_P(\zeta):={2\over (1-|\zeta|^2)^2} i d\zeta\wedge d\overline\zeta,\qquad\zeta\in\D,\quad \text{where}\ i:=\sqrt{-1}.  $$
  
A leaf $L$  of the  foliation is  said to be  {\it hyperbolic} if
it  is a   hyperbolic  Riemann  surface, i.e., it is  uniformized   by 
$\D.$   For any point $x\in X\setminus E,$  let $L_x$  be the   leaf passing  through $x.$   The   foliation   is  said to be {\it hyperbolic} if  
   all its leaves   are   hyperbolic.
For a hyperbolic  leaf $L_x,$
  consider a universal covering map
\begin{equation}\label{e:covering_map}
\phi_x:\ \D\rightarrow L_x\qquad\text{such that}\  \phi_x(0)=x.
\end{equation}
 This map is
uniquely defined by $x$ up to a rotation on $\D$. 
Then, by pushing   forward  the Poincar\'e metric $g_P$
on $\D$  
  via $\phi_x,$ we obtain the  so-called {\it Poincar\'e metric} on $L_x$ which depends only on the leaf. 
  The latter metric is given by a positive $(1,1)$-form on $L_x$  that we also denote by $g_P$ for the sake of simplicity.   

 For simplicity  we still denote by   $g_X$  the  Hermitian  metric on leaves of the foliation $(X\setminus E,\Lc)$  induced by  the ambient Hermitian metric  $g_X.$
  Consider the  function $\eta:\ X\setminus E\to [0,\infty]$ defined by
  $$
  \eta(x):=\sup\left\lbrace \|d\phi(0)\| :\ \phi:\ \D\to L_x\ \text{holomorphic such that}\ \phi(0)=x  \right\rbrace.
  $$
  Here, for the  norm of the  differential $d\phi$ we use the Poincar\'e metric on $\D$ and the Hermitian metric
  $g_X$ on $L_x.$ Clearly, we have 
  \begin{equation}\label{e:extremum}
   \eta(x)=\|d\phi_x(0)\|\qquad\text{for}\qquad  x\in X\setminus E.
  \end{equation}
  We  recall the following relation  between  $g_X$   and  the Poincar\'e metric $g_P$ on leaves
\begin{equation}
\label{e:relation_Poincare_Hermitian_metrics}
g_X(y)|_{L_x}=\eta^2(y) g_P(y)\qquad\text{for}\qquad y\in L_x,\  x\in X\setminus E.
\end{equation}

In   \cite{DinhNguyenSibony14b} the following class of foliations is  introduced. 

\begin{definition}\label{D:Brody}\rm 
A   foliation $\Fc=(X,\Lc,E)$   is  said  to be {\it Brody hyperbolic}  if  there is  a constant  $c>0$  such that
$\eta(x) \leq  c$ 
for all $x\in X\setminus E.$ 
\end{definition}

 \begin{remark}\rm  \label{R:hyperbolicty}
Note that  if the foliation $\Fc$ is Brody hyperbolic then it is hyperbolic. But the converse statement does not hold in general.

The  Brody hyperbolicity is  equivalent to 
the non-existence  of  holomorphic non-constant maps
$\C\rightarrow X$  such that out of  $E$ the image of $\C$ is  locally contained in a leaf, 
see \cite[Theorem 15]{FornaessSibony08}.
If $\Fc$ admits no directed positive closed current, then it is  Brody hyperbolic. 
\end{remark}

As  a consequence of \eqref{e:relation_Poincare_Hermitian_metrics}, the {\it  Poincar\'e norm}  $|df|_P$ of the differential $df$  for
  a differentiable function $f:\ L_x\to\C$   is given by   
$$|df|_P(y):=  \eta(y) |df(y)|\qquad\text{for}\qquad y\in L_x,
$$
 where $|df(y)|$ denotes the Euclidean norm of $df(y).$
\subsection{Poincar\'e laplacian $\Delta_P$ and the heat diffusions, directed positive harmonic currents vs harmonic measures}
\label{SS:background}
 A {\it (directed) $p$-form} (resp. a {\it (directed) $(p,q)$-form}) on $\Fc$ can be seen on the flow box
$\U\simeq \B\times\T$ as
a $p$-form  (resp.  $(p,q)$-form) on $\B$ depending on the parameter $t\in \T$. 
For $0\leq p\leq 2$ (resp. for $0\leq p,q\leq 1$),  denote by $\Dc^p_l(\Fc)$  (resp. $\Dc^{p,q}_l(\Fc)$) the space of 
$p$-forms (resp.  $(p,q)$-form) $f$ with compact support in $X\setminus E$ satisfying the following property: 
$f$ restricted to each flow box $\U\simeq \B\times\T$ is a 
$p$-form (resp.  $(p,q)$-form) of class $\Cc^l$ on the plaques whose coefficients and
all their derivatives up to order $l$ depend continuously on the
plaque. The norm $\|\cdot\|_{\Cc^l}$ on
this space is defined as in the case of real manifold using a locally finite
atlas of $\Fc$. We also define $\Dc^p(\Fc)$ (resp.  $\Dc^{p,q}(\Fc)$)  as the intersection of
$\Dc^p_l(\Fc)$ (resp.  $\Dc^{p,q}_l(\Fc)$) for $l\geq 0$.
  In particular, a
sequence $f_j$ converges to $f$ in $\Dc^p(\Fc)$ (resp. in  $\Dc^{p,q}(\Fc)$) if
these forms are supported in a fixed compact set of $X\setminus E$  and if
$\|f_j-f\|_{\Cc^l}\rightarrow 0$ for every $l$.
A {\it (directed) current of degree $p$ ({\rm or equivalently,}  of dimension $2-p$)} 
on $\Fc$ is a continuous
linear form on the space $\Dc^{2-p}(\Fc)$ 
with values in $\C$. 
  We often write for short $\Dc(\Fc)$ instead of  $\Dc^0(\Fc).$
  A directed $(p+q)$-current is said to be  of
{\it bidegree $(p,q)$  ({\rm or equivalently,}  of bidimension $(1-p,1-q)$)} if it vanishes on forms of bidegree $(1-p',1-q')$ for $(p',q')\not=(p,q).$

   A form $f \in \Dc^{1,1}(\Fc)$ is  said to be {\it positive} if its restriction to every plaque
 is  a  positive measure  in the usual sense.
 \begin{definition}\label{D:harmonic_current}  {\rm (Garnett \cite{Garnett}, see also  Sullivan \cite{Sullivan}).}
\rm  Let $T$ be a directed current of bidegree $(1,1)$ on $\Fc.$ 
    
  $\bullet$  $T$ is  said to be  {\it positive} if  $T(f)\geq 0$ for all positive forms $f\in \Dc^{1,1}(\Fc).$
    
$\bullet$ $T$ is  said to be  {\it   harmonic}  if  $\ddc T=0$ in the  weak sense (namely,  $T(\ddc f)=0$ for all functions  $f\in  \Dc(\Fc)$).
\end{definition}

Let $\U$ be any
flow box of $\Fc$ outside the singularities and denote by $V_\alpha$ the plaques of $\Fc$ in
$\U$ parametrized by $\alpha$ in some transversal $\Sigma$ of $\U.$ 
On the flow box $\U,$ a   positive harmonic  current $T$ directed by $\Fc$ (or equivalently, tangent to $\Fc$)  has
the form
\begin{equation}\label{e:decomposition}
T|_\U=\int_{\alpha\in\Sigma} h_\alpha [V_\alpha] d\nu(\alpha),
\end{equation}
where $h_\alpha$ is  a positive harmonic function on  $V_\alpha,$ and $[V_\alpha]$ denotes the current of intergration on the plaque $V_\alpha,$ and $\nu$ is a Radon measure on $\Sigma$  (see e.g. \cite[Proposition 2.3]{DinhNguyenSibony12}).

Let  $\Fc=(X,\Lc,E)$ be  a hyperbolic  foliation.  
     The leafwise Poincar\'e metric
$g_P$  induces  the corresponding 
Laplacian $\Delta_P$  on leaves such that
\begin{equation}\label{e:Laplacian}
\ddc f|_{L_x}={1\over \pi}\Delta_P f \cdot g_P,\qquad \text{on $L_x$ for all}\ f\in \Dc(\Fc).
\end{equation}
A positive finite  Borel measure  $\mu$ on $X$  which does  not  give mass to $E$  is said  to be
{\it  harmonic}  if
$$
\int_X  \Delta_P f d\mu=0
$$ 
 for all  functions  $f\in \Dc(\Fc).$

 For  every point  $x\in X\setminus E,$
 consider  the   {\it heat  equation} on $L_x$
 $$
 {\partial p(x,y,t)\over \partial t}=\Delta_{P,y} p(x,y,t),\qquad  \lim_{t\to 0} p(x,y,t)=\delta_x(y),\qquad   y\in L_x,\ t\in \R_+.
 $$
Here   $\delta_x$  denotes  the  Dirac mass at $x,$ $\Delta_{P,y}$ denotes the  Laplacian  $\Delta_P$ with respect to the  variable $y,$
 and  the  limit  is  taken  in the  sense of distribution, that is,
$$
 \lim_{t\to 0+
}\int_{L_x} p(x,y,t) f(y) g_P( y)=f(x)
$$
for  every  smooth function  $f$   compactly supported in $L_x.$   

The smallest positive solution of the  above  equation, denoted  by $p(x,y,t),$ is  called  {\it the heat kernel}. Such    a  solution   exists   because  $(L_x,g_P)$ is
complete and   of bounded  geometry  (see, for example,  \cite{CandelConlon2,Chavel}). 
 The  heat kernel   gives  rise to   a one  parameter  family $\{D_t:\ t\geq 0\}$ of  diffusion  operators    defined on bounded Borel measurable functions  on $X\setminus E:$
 \begin{equation}\label{e:diffusions}
 D_tf(x):=\int_{L_x} p(x,y,t) f(y) g_P (y),\qquad x\in X\setminus E.
 \end{equation}
 We record here  the  semi-group property  of this  family: 
\begin{equation}\label{e:semi_group}
D_0=\id\quad\text{and}\quad   D_t \mathbf{1} =\mathbf{1}\quad\text{and}\quad D_{t+s}=D_t\circ D_s \quad\text{for}\ t,s\geq 0,
\end{equation}
where $\mathbf{1}$ denotes the function which is identically equal to $1.$

We also  denote by  $\Delta_P$ the  Laplacian on the Poincar\'e disc $(\D,g_P),$ that is, for every function $f\in \Cc^2(\D),$
\begin{equation}\label{e:Delta_P}
{1\over \pi}(\Delta_P  f) g_P=\ddc f \qquad\text{on}\qquad \D.
\end{equation}
For every function $f\in \Cc^1(\D),$ we also denote by $|df|_P$  the length  of  the differential $df$ with respect to $g_P,$
that is,   
\begin{equation}\label{e:df_P}|df|_P=|df|\cdot g_P^{-1/2}\qquad \text{on}\qquad \D,
 \end{equation}
 where $|df|$ denotes the Euclidean norm of $df.$  Let $\dist_P$ denote the Poincar\'e distance on $(\D,g_P).$

Using   the map $\phi_x:\  \D\to L_x$ given in  (\ref{e:covering_map}), the following identity relates the diffusion operators in  $L_x$ and those  in
the  Poincar\'e disc $(\D,g_P):$ For   $x\in X$ and for every  bounded measurable  function $f$ defined on $L_x,$  
\begin{equation}\label{e:commutation}
 D_t(f\circ \phi_x)=(D_tf)\circ\phi_x, \qquad \textrm{on $L_x$  for all $t\in\R^+.$} 
\end{equation}
Here $\{D_t:\ t\geq 0\}$ on the left-hand side is the family of  the  heat diffusion associated to the  Poincar\'e disc,  see \cite[Proposition 2.7]{NguyenVietAnh17} for a proof.

Recall that a positive finite  measure  $\mu$ on the $\sigma$-algebra of Borel sets in $X$ with $\mu(E)=0$  is  said  to be 
{\it ergodic} if for every  leafwise  saturated  Borel measurable set $Z\subset X,$
   $\mu(Z)$ is  equal to either $\mu(X)$ or $0.$   A   directed positive harmonic  current $T$ is said to be {\it extremal}
   if  $T=T_1+T_2$ for   directed positive harmonic  current $T_1,T_2$ implies that $T_1=\lambda T$  for some $ \lambda \in  [0, 1].$
   
 The following    result which gives the link between directed positive harmonic currents, directed positive $\ddc$-closed currents and  harmonic measures 
 (see  e.g.  Theorem 2.9 and Proposition 2.17 in  \cite{NguyenVietAnh19}, see also \cite{DinhNguyenSibony12}).
 
 \begin{proposition}\label{P:harmonic_currents_vs_measures} 
Let $\Fc=(X,\Lc,E)$ be  a hyperbolic  foliation with  linearizable singularities  $E$ in a compact complex surface $X.$ 
\begin{enumerate}
\item Every  directed positive  harmonic  current extends, by trivial  extension across $E,$   to a   $\ddc$-closed  current on $X.$
In other words,  directed positive harmonic currents  are equivalent to   directed positive $\ddc$-closed currents. 
\item  
  The  relation  $T\mapsto \mu$ given  in \eqref{e:mu} is  a one-to-one  correspondence between the  convex  cone of  positive harmonic  currents $T$ directed by $\Fc$ and  
the convex  cone of  harmonic  measures $\mu$. 
\item $T$ is  extremal if and only if    $\mu$ given  in \eqref{e:mu} is  ergodic.
\item  Each  harmonic    measure $\mu$  is $D_t$-invariant, i.e, 
$$ \int_X  D_tf d\mu=\int_X fd\mu,  \qquad  f\in L^1(X,\mu),\  t\in\R^+.   $$
\end{enumerate}
\end{proposition}
 


 \subsection{Local model for hyperbolic  singularities,  regular and singular flow boxes}
 \label{SS:local_model}
 
   To study $\Fc$ near  a hyperbolic  singularity $a,$ we use the following {\bf local model} introduced in \cite{DinhNguyenSibony14a}.
In this model, a neighborhood of $a$ is identified with the  bidisc $\D^2,$ and the restriction of $\Fc$  to $\D^2,$  i.e., the leaves of $(\D^2,\Lc,\{0\})$  coincide with 
the restriction to $\D^2$ of the   
 integral curves of a  vector field 
 \begin{equation}\label{e:Z}Z(z,w) = z {\partial\over \partial z}
+ \lambda w{\partial\over \partial w}\quad\text{with some}\quad \lambda\in\C\setminus \R. 
\end{equation}
 
For $x=(z,w)\in \D^2\setminus\{0\}$, define the holomorphic map  $\psi_x:\C\rightarrow\C^2\setminus\{0\}$ 
\begin{equation}\label{e:leaf_equation}
\psi_x(\zeta):=\Big( ze^{i\zeta},we^{i\lambda\zeta}\Big)\quad \mbox{for}\quad \zeta\in\C.
\end{equation}
It is easy to see that $\psi_x(\C)$ is the integral curve of $Z$ which contains
$\psi_x(0)=x$.
Write $\zeta=u+iv$ with $u,v\in\R$. The domain $\Pi_x:=\psi_x^{-1}(\D^2)$ in $\C$ is defined by the inequalities
$$(\Im \lambda)u+(\Re\lambda)v >\log|w|\qquad \text{and}\qquad v > \log|z| .$$
So, $\Pi_x$ defines a sector $\S_x$ in $\C.$ It contains $0$ since $\psi_x(0)=x$.  
The leaf of $\Fc$ through $x$ contains  the  Riemann surface 
\begin{equation}\label{e:Riemann_surface_L_x}
\widehat L_x:=\psi_x(\Pi_x)\subset L_x.
\end{equation}
In particular, the leaves in a singular flow box are parametrized   using holomorphic maps $\psi_x:\Pi_x\to L_x.$

 Now  let  $\Fc=(X,\Lc,E)$ be a  foliation on a   Hermitian compact complex surface $(X,g_X).$
Assume as usual that $E$  is  finite  and   all  points of $E$  are  linearizable. 
Let $\dist$ be the  distance on $X$ induced  by the ambient  metric $g_X.$
We only consider flow boxes which are biholomorphic to $\D^2$. A {\it regular flow box} $\U$ is a flow box with  foliated chart $\Phi:\ \U\to \B\times\Sigma$ outside the singularities, where 
$\B$ and $\Sigma$ are open sets in $\C.$ 
 For each  $\alpha\in\Sigma,$ the  Riemann surface $V_\alpha:=\Phi^{-1}( \B\times\{\alpha\})$ is called a plaque  of $\U.$ 
{\it Singular flow boxes} are identified to their models $(\D^2,\Lc,\{0\})$ as 
described above. 
 For  $\U:=\D^2$ and  $s>0,$  let
 $s\U:= (s\D)^2.$
For each singular point $a\in E$, we fix a singular flow box $\U_a$ such that $2\U_a\cap 2\U_{a'}=\varnothing$ if $a,a'\in E$ with $a\not=a'$.
We also cover $X\setminus \cup_{a\in E} \U_a$ by a finite 
number of regular flow boxes $(\U_q)$ such that  each $\overline \U_q$ is contained in a larger regular flow box $\U'_q$ with $\U'_q\cap \cup_{a\in E} (1/2)\U_a =\varnothing.$ 
Thus  we obtain a  finite  cover $\Uc:=(\U_p)_{p\in I} $ of $X$  consisting  of  regular  flow boxes $(\U_p)_{p\in I\setminus E}$ and  singular ones $(\U_a)_{a\in E}.$  

We often suppose without loss of generality that  the ambient metric $g_X$ coincides with the standard Euclidean metric on each singular flow box $2\U_a\simeq 2\D^2,$ $a\in E.$ For $x=(z,w)\in\C^2,$ let
 $\|x\|:=\sqrt{|z|^2+|w|^2}$ be the  standard Euclidean norm of $x.$ Recall  that
 $\lof(\cdot):=1+|\log(\cdot)|.$

Recall from \cite[Proposition 3.3]{DinhNguyenSibony14b} (see also  \cite[Lemma 2.4]{NguyenVietAnh18b}) the following   precise estimate on the  function $\eta$  introduced  in  (\ref{e:relation_Poincare_Hermitian_metrics}).   

\begin{lemma} \label{L:Poincare}
We keep  the above  hypotheses
and notation.  Suppose  in addition that  $\Fc$ is Brody hyperbolic. Then there exists  a  constant $c>1$  with the following properties.
\begin{enumerate} 
\item $\eta \leq c$ on $X$, $\eta\geq c^{-1}$ outside the singular flow boxes $\cup_{a\in E}{1\over 4}\U_a$ and 
$$c^{-1} \cdot s \lof s \leq\eta(x)  \leq c \cdot s \lof s$$
for $x\in X\setminus E$  and $s:=\dist(x,E).$  
\item  For every  $x$ in  a singular box  which is identified with $(\D^2,\Lc,\{0\})$   as above,
 for every $\zeta\in\Pi_x,$
 $$
    c^{-1}\cdot {i d\zeta\wedge d\bar\zeta \over (\lof (\psi_x(\zeta)))^2}\leq (\psi_x^* g_P) (\zeta)\leq  c\cdot {i d\zeta\wedge d\bar\zeta \over (\lof (\psi_x(\zeta)))^2}.
 $$
 \end{enumerate}
\end{lemma}

We  revisit  the  local model by describing  its special flow box  structure.
  Notice that
if we flip $z$ and $w,$ we replace $\lambda$ by $\lambda^{-1}.$ Since $(\Im\lambda)( \Im\lambda^{-1})<0,$ we
may assume   below that the axes are chosen so that $\Im \lambda > 0.$
Consider the ring $\A$ defined by
$$
\A := \left\lbrace  \alpha\in \C,\ e^{-2\pi \Im \lambda } < |\alpha| \leq 1 \right\rbrace.$$
Define also the sector $\S$   by
$$
\S := \left\lbrace \zeta = u + iv \in \C,\ v > 0\quad\text{and}\quad (\Im \lambda)u + (\Re \lambda)v > 0\right\rbrace.
$$
Note that the sector $\S$ is contained in the upper half-plane $\H := \{u+iv,\ v > 0\}.$ 
For $\alpha\in  \C^* ,$ consider 
the following holomorphic map $\tilde\psi_\alpha:\ \C\to(\C^*)^2:$ 
\begin{equation}\label{e:leaf_equation_bis}\tilde\psi_\alpha:= \big( e^{i(\zeta +\log{ |\alpha|/\Im \lambda})},\alpha e^{i\lambda (\zeta +\log {|\alpha|/\Im \lambda})}\big)\quad \text{for}\quad \zeta = u + iv \in \C.
\end{equation}
Note that the map $\tilde\psi_\alpha$  is injective because $\lambda \not\in\R.$
Let  $\mathcal  L'_\alpha$ be the image  of  $\tilde\psi_\alpha.$  This is a Riemann surface immersed in $\C^2.$ 
Note also that
\begin{equation}\label{e:|z|}|z| = e^{-v} 
\qquad\text{and}\qquad   |w| = e^{-(\Im \lambda) u-(\Re \lambda) v}.
\end{equation} It is easy to check
the following properties
\begin{enumerate}
\item $\mathcal L'_\alpha$ is tangent to the vector field $Z$ given in \eqref{e:Z} and is a submanifold of $\C^{*2} .$
\item  $\mathcal L'_{\alpha_1}$ is equal to $\mathcal L'_{\alpha_2}$ if $\alpha_1 /\alpha_2 = e^{2ki\lambda \pi}$ for some $k\in\Z$ and they are
disjoint otherwise. In particular, $\mathcal L'_{\alpha_1}$ and $\mathcal L'_{\alpha_2}$ are disjoint if $\alpha_1 , \alpha_ 2 \in \A$
and $\alpha_1 \not = \alpha_2 .$
\item  The union of $\mathcal L'_\alpha$ is equal to $\C^{*2}$ for $\alpha \in \C^* ,$ and then also for $\alpha\in  \A.$
\item  The intersection $\mathcal L_\alpha :=\mathcal  L'_\alpha \cap \D^2$ of $\mathcal L'_\alpha$ with the unit bidisc $\D^2$ is given by
the same equations as in the definition of $\mathcal L'_\alpha$  but with $\zeta \in \S.$ Moreover,
$\mathcal L_\alpha$ is a connected submanifold of $\D^{*2} .$ In particular, it is a leaf of $\Lc \cap \D^2 .$
\end{enumerate}

 
 \subsection{Measure theory on  sample-path spaces}
 \label{SS:Wiener}
  
In this  subsection
 we follow the   presentation  given in Section 2.2, 2.4 and 2.5 in  \cite{NguyenVietAnh17}, which are,  in turn,  inspired by
 Garnett's theory   of leafwise  Brownian  motion in \cite{Garnett} (see also \cite{Candel03,CandelConlon2}).
 This   exposition is  equivalent  to that given in \cite[Subsection 2.4]{NguyenVietAnh18b}.
 
 We first recall the construction of the Wiener measure $W_0$ on the  Poincar\'e  disc $(\D,g_P).$
 Let $\Omega_0$ be  the space consisting of  all continuous  paths  $\omega:\ [0,\infty)\to  \D$  with $\omega(0)=0.$
A {\it  cylinder  set (in $\Omega_0$)} is a 
 set of the form
$$
C=C(\{t_i,B_i\}:1\leq i\leq m):=\left\lbrace \omega \in \Omega_0:\ \omega(t_i)\in B_i, \qquad 1\leq i\leq m  \right\rbrace.
$$
where   $m$ is a positive integer  and the $B_i$'s are Borel subsets of $\D,$ 
and $0< t_1<t_2<\cdots<t_m$ is a  set of increasing times.
In other words, $C$ consists of all paths $\omega\in  \Omega_0$ which can be found within $B_i$ at time $t_i.$
Let $\Ac_0$ be the  $\sigma$-algebra on $\Omega_0$  generated  by all  cylinder sets.
 For each cylinder  set  $C:=C(\{t_i,B_i\}:1\leq i\leq m)$ as  above, define
\begin{equation}\label{eq_formula_W_x_without_holonomy}
W_x(C ) :=\Big (D_{t_1}(\chac_{B_1}D_{t_2-t_1}(\chac_{B_2}\cdots\chac_{B_{m-1}} D_{t_m-t_{m-1}}(\chac_{B_m})\cdots))\Big) (x),
\end{equation}
where,  $\chac_{B_i}$
is the characteristic function of $B_i$ and $D_t$ is the diffusion operator
given  by  (\ref{e:diffusions}) where  $p(x,y,t)$  is   the heat kernel    of the Poincar\'e disc $(\D,g_P).$ 
 It is  well-known that $W_0$ can be   extended   to a unique probability measure on $(\Omega_0,\Ac_0).$
This  is the {\it canonical  Wiener measure}  at $0$  on the Poincar\'e disc.

Let $\Fc=(X,\Lc,E)$ be  a hyperbolic Riemann surface foliation with  the set of singularities $E.$  Consider the leafwise Poincar\'e metric $g_P.$
Recall from  Introduction that   $\Omega:=\Omega(\Fc) $  is  the space consisting of  all continuous  paths  $\omega:\ [0,\infty)\to  X\setminus E$
with image fully contained  in a  single   leaf. This  space  is  called {\it the sample-path space} associated to  $\Fc.$
  Observe that
$\Omega$  can be  thought of  as the  set of all possible paths that a 
Brownian particle, located  at $\omega(0)$  at time $t=0,$ might  follow as time  progresses.
For each $x \in  X\setminus E,$
let $\Omega_x=\Omega_x(\Fc)$ be the  space  of all continuous
leafwise paths  starting at $x\in X\setminus E$ in $\Fc,$ that is,
$$
\Omega_x:=\left\lbrace   \omega\in \Omega:\  \omega(0)=x\right\rbrace.
$$
For each $x\in X\setminus E,$    the  following  mapping 
\begin{equation}\label{e:Omega_0_vs_Omega_x}
\Omega_0\ni\omega \mapsto  \phi_x\circ\omega\quad
 \text{maps}\quad  \Omega_0\quad\text{bijectively onto}\quad \Omega_x,
 \end{equation}
 where   $\phi_x:\D\to L_x$ is given in (\ref{e:covering_map}).   
Using this bijection  
we obtain   a  natural $\sigma$-algebra    $ \Ac_x$ on  the  space $\Omega_x,$ and  a natural    probability (Wiener) measure $W_x$  on $\Ac_x$  as follows:
\begin{equation}\label{e:W_x}
 \Ac_x:=\{ \phi_x\circ A:\  A\in \Ac_0\}\quad\text{and}\quad    W_x(\phi_x\circ A):=W_0(A),\qquad   A\in  \Ac_0,
\end{equation}
   where $\phi_x\circ A:= \{\phi_x\circ \omega:\ \omega\in A \}\subset \Omega_x.$

For any   function $F\in L^1(\Omega_x,\Ac_x,W_x),$
 the {\it  expectation} of $F$ at $x$ is  the  number
\begin{equation}\label{e:expectation}
\E_x[F]:=\int_{\Omega_x} F(\omega)dW_x(\omega).
\end{equation}
  It is well-known (see \cite[Proposition C.3.8]{CandelConlon2}) that for any  bounded Borel measurable  function $f$ on $L_x,$
\begin{equation}\label{e:expectation_vs_diffusion}
\E_{x}[f(\bullet(t))]=(D_tf)(x),\qquad  t\in\R^+,
\end{equation} 
where $f(\bullet(t))$ is  the  function  given by $\Omega\ni\omega\mapsto f(\omega(t)).$

\subsection{Holonomy cocycle}

Now  we recall  from   \cite[Subsection 2.5 and Section 3]{NguyenVietAnh18b}   
the holonomy cocycle   as well as its properties of a hyperbolic foliation $\Fc=(X,\Lc,E)$ on a Hermitian  complex surface $X.$
For each point $x\in X\setminus E,$ let $\Tan_x(X)$ (resp.  $\Tan_x(L_x)\subset \Tan_x(X)$) be    the tangent space  of $X$ (resp. $L_x$) at $x.$
 For every transversal    $S$ at a point $x$ (that is, $x\in S$), let  $\Tan_x(S)$ denote the tangent space of $S$ at $x.$  

Now fix a point $x\in X\setminus E$ and a path $\omega\in\Omega_{x}$  and a time $t\in \R^+,$ and  let $y:=\omega(t).$
Fix  a transversal $S_x$ at $x$ (resp. $S_y$  at $y$) such that the complex line
$\Tan_{x}(S_x)$ is the  orthogonal complement  of the complex line $\Tan_{x}(L_{x})$ in the Hermitian space $(\Tan_x(X),(g_X)_x)$
(resp.        $\Tan_{y}(S_y)$ is the  orthogonal complement  of $\Tan_{y}(L_{y})$ in $(\Tan_{y}(X),(g_X)_y)$). 
Let  $\hol_{\omega,t}$  be the  holonomy map along the path  $\omega|_{[0,t]}$ from   an open  neighborhood of $x$ in $S_x$   onto
 an open  neighborhood of $y$ in $S_y.$     The  derivative  $D \hol_{\omega,t}:\
\Tan_{x}(S_x)\to \Tan_{y}(S_y) $ induces 
the so-called {\it holonomy cocycle} $\mathcal H:\ \Omega\times\R^+\to\R^+$ given by  $$\mathcal H(\omega,t):=
  \|D \hol_{\omega,t}(x)\|.$$ 
 The last  map depends  only on the path  $\omega|_{[0,t]},$ in fact, it
depends only on the  homotopy class of  this path.
 In particular, it is  independent of   the choice  of transversals  $S_x$ and $S_y.$ 
 We see easily that
 \begin{equation}\label{e:hol_cocycle_dist}
\mathcal H(\omega,t) =\lim_{z\to x,\ z\in S_x} \dist(\hol_{\omega,t}(z), y)/ \dist(z,x).
  \end{equation}

 The  following result gives an explicit  expression for  $\mathcal H$ near a singular point $a$ 
 using the local model  $(\D^2,\Lc,\{0\})$ introduced in  Subsection   \ref{SS:local_model}.  
\begin{lemma}\label{L:holonomy} {\rm (\cite[Proposition 3.1]{NguyenVietAnh18b})} Let $\D^2$ be endowed with the  Euclidean metric.
For each  $x=(z,w)\in\D^2,$ consider the function  $\Phi_x:\  \Pi_x\to\R^+$   as follows.
For $\zeta\in \Pi_x,$ consider a path $\omega\in \Omega$ (it always exists since $\Pi_x$ is convex and $0\in\Pi_x$ as $x\in\D^2$) such  that
$$
\omega(t)= \psi_x(t\zeta)= (z e^{i \zeta t},we^{i\lambda\zeta t})\subset \D^2
$$
for all $t\in [0,1]$ (see \eqref{e:leaf_equation} above). 
Define $\Phi_x(\zeta):=  
\mathcal H(\omega,1).$
Then 
$$
\Phi_x(\zeta)=|e^{i \zeta}||e^{i \lambda\zeta}|{\sqrt{| z|^2+|\lambda w|^2}\over \sqrt{ | ze^{i\zeta}|^2+|\lambda we^{i \lambda\zeta}|^2           }}. 
$$
\end{lemma}

\subsection{Chern curvature, Chern class} \label{SS:Chern}

Let $(L , h)$ be a singular Hermitian holomorphic line bundle on $X.$
If $e_L$ is a holomorphic frame of $L$ on some open set $U \subset X,$
then the function $\varphi $ defined  by  $|e_L |^2_h = \exp{( -2\varphi)} $  is called the {\it local weight} of the metric $h$ with respect to $e_L .$ 
If   the local weights $\varphi$ are in $L_\loc^1(U),$  then {\it (Chern) curvature current} of $(L,h)$ denoted by $c_1 (L , h)$ 
is  given by $c_1 (L , h)|_U = \ddc \varphi.$  This is   a $(1,1)$-closed current. Its class in $H^{1,1}(X)$ is  called the Chern class of $L.$
If  we fix a  smooth  Hermitian metric $h_0$ on $L,$ then every  singular metric $h$ on $L$ can be written $h=e^{-2\varphi} h_0$ for some  function $\varphi.$
We say that $\varphi$ is the {\it global weight} of $h$ with respect to $h_0.$ Clearly,  $c_1(L,h)=c_1(L,h_0)+\ddc\varphi.$

\subsection{Transversal metric, transversal form, change of variables}
\label{SS:holonomy_cocycle}
  Let  $\Fc=(X,\Lc,E)$  be a holomorphic  foliation on a Hermitian  complex surface $X.$
 Consider the normal bundle $\Nor(\Fc) = \Tan(X)/\Tan(\Fc)$ of $\Fc.$
 For $x\in X\setminus E$ and a  vector $u_x\in \Tan_x(X),$ let 
 $[u_x]$ denotes its class in $\Nor_x(\Fc).$
 We also identify
 $[u_x]$ with the  set $u_x +\Tan_x(\Fc)\subset \Tan_x(X)$
  Note that   a (local) smooth  section  $w$ of  $\Nor(\Fc)$ can be locally written as  $w_x=[u_x]$ for some
  smooth vector field $u.$
 
 Consider the  following  metric 
$g^\perp_X$ on the normal bundle $\Nor(\Fc):$  
 \begin{equation}\label{e:tran_metric}
  \|w_x\|_{g^\perp_X}:=\min_{u_x\in [w_x]} \|u_x\|_{g_X},\qquad\text{for}\qquad  w_x\in \Nor(\Fc)_x,\ x\in X\setminus E.
 \end{equation}
 Note  that $u_x$ achieving  the minimum in  \eqref{e:tran_metric} is uniquely determined by $[w_x].$
 $g^\perp_X$ is  called   {\it  transversal metric} associated with $\Fc$ and  the ambient metric $g_X.$

 Fix a  smooth Hermitian metric $g_0$ on $\Nor(\Fc).$ There is  a  global weight  function $\varphi$ on $X$  such that
 $  g^\perp_X=e^{-2\varphi} g_0. $
 Let $w$ be  a (smooth) section  of $\Nor(\Fc)$ over a  small open  subset $\U\subset X$   such that
 $ \|w_x\|_{{\hat g}^\perp_X}=1$ for $x\in \U.$ Let $\beta$ be the dual  form of $w$ with respect to $g_x,$ defined by
 $$
 \langle\beta_x, v_x\rangle=g_X(v_x,u_x)\quad\text{for all}\quad  x\in \U\quad\text{and}\quad v_x\in\Tan_x(X) .
 $$
 Here by shrinking $\U$  if necessary, we may assume that $u$ is a vector field on  $\U$ such that $w_x=[u_x]$ for $x\in U.$
 The {\it transversal form} associated with $\Fc$ and  the ambient metric $g_X$ is the smooth  positive $(1,1)$-form  denoted by $\omega^\perp_X$
 defined on $\U$ by
 \begin{equation}\label{e:tran_form}
 \omega^\perp_X:= i\beta\wedge \bar\beta.
 \end{equation}
 Patching this form over all  such open sets $\U,$   
 we obtain   a $(1,1)$-smooth form  $\omega^\perp_X$ well-defined  on $X\setminus E.$

 
 In each regular flow box $\U$ with  foliated chart
    $\Phi:\ \U\to \B\times \Sigma$ with $0\in\B,$ consider  the following  volume form $\Upsilon$ on $\Sigma$
    \begin{equation}\label{e:Upsilon}
\Upsilon(\alpha):= 
(\Phi_*\omega^\perp_X)(0,\alpha)\quad\text{for}\quad \alpha\in\Sigma.  
\end{equation} 
Here  $\B\times \Sigma$ is  an open subset of $\C\times \C.$
    Then there exists   a unique   function $\varphi_\U:\  \U\to \R$  such that
\begin{equation}\label{e:local_varphi}
(id\zeta\wedge d\bar \zeta)\wedge \Phi_*(\omega^\perp_X)(\zeta,\alpha)=\exp{\big(-2\varphi_\U(\Phi^{-1}(\zeta,\alpha))\big)}\cdot
(id\zeta\wedge d\bar \zeta)\wedge \Upsilon(\alpha)\quad\text{for}\quad (\zeta,\alpha)\in\B\times \Sigma.
\end{equation}
In fact, if we  assume  that $\B$ is a simply connected domain in $\C,$ for example, $\B$ is a disc in $\C$  centered at $0,$  then for $x\in\U$ we  write $(\zeta,\alpha)=\Phi(x),$ and  obtain that  
\begin{equation}\label{e:varphi_U}\varphi_\U(x)=-\ln\mathcal H(\omega,1),
\end{equation}  where $\mathcal H$ is the  holonomy cocycle  (see \eqref{e:hol_cocycle_dist}), $\omega: [0,1]\to  L_x$   
 is  a path in a  plaque of $\U$  going from    $y:=\Phi^{-1}(0,\alpha)$  to $x.$   
Note that $\varphi_\U(0,\alpha)=0$ for $\alpha\in\Sigma.$ 

We infer from \eqref{e:tran_form}, \eqref{e:Upsilon} and \eqref{e:local_varphi}  the  following change of variables for integrals with directed forms:
\begin{equation}\label{e:change-variables}
 \langle h,\omega^\perp_X \rangle=\int_\Sigma\Big(\int_{\B} \exp{\big(-2\varphi_\U(\Phi^{-1}(\zeta,\alpha))\big)}(\Phi_*h)(\zeta,\alpha)\Big)  \Upsilon(\alpha)\quad\text{for}\quad h\in\Dc^{1,1}(\Fc).
\end{equation}

In each  singular flow box $\U_a$ with $a\in \E$ we use  the local model $\U_a\simeq \D^2$ and  consider
$\Phi:\ \U_a\to \B\times\Sigma=\S\times \A$ which is the inverse map of $\S\times \A\ni (\zeta,\alpha)\mapsto \tilde\psi_\alpha(\zeta)\in (\D\setminus \{0\})^2,$
where  $\tilde\psi_\alpha$ is  given by
\eqref{e:leaf_equation_bis}.
Consequently, using \eqref{e:Upsilon} and  \eqref{e:varphi_U}, we can find a volume form $\Upsilon$ on $\A$ and  a  function $\varphi_a:\  (\D\setminus \{0\})^2\to\R$  such that
\begin{equation}\label{e:local_varphi_bis}
(id\zeta\wedge d\bar \zeta)\wedge \Phi_*(\omega^\perp_X)(\zeta,\alpha)=\exp{\big(-2\varphi_a(\zeta,\alpha)\big)}\cdot 
(id\zeta\wedge d\bar \zeta)\wedge \Upsilon(\alpha)\quad\text{for}\quad (\zeta,\alpha)\in\S\times \A.
\end{equation}
We also  obtain a similar formula as \eqref{e:change-variables} for  the singular flow box $\U_a.$ 

 
 \subsection{Specializations  $\kappa_x$ and the 
 curvature density $\kappa$}
 

We recall some notions and results from \cite[Section 9.1]{NguyenVietAnh17}.
 Fix a  point $x\in X$   and let
   $\phi_{x}:\ \D\to L=L_{x}$   be the universal covering map    given in (\ref{e:covering_map}).    
    Consider  
      the  function  $\kappa_x:\  \D\to \R$ defined by
\begin{equation}\label{e:specialization}
\kappa_x(\zeta):= \log \| {\mathcal H}(\phi_x\circ\omega,1) \|,\qquad  \zeta \in \D,
\end{equation}
where  $\omega\in \Omega_0$ is any path  such that  $\omega(1)=\zeta.$ 
This   function is  well-defined   because $\mathcal H(\phi_x\circ \omega,t)$ depends only on the homotopy class of the path $\omega|_{[0,t]}$ and $\D$ is   simply connected.
Following \cite{NguyenVietAnh17},
  $\kappa_x$ is  said  to be  the {\it specialization}  of the holonomy cocycle  $\mathcal H$ at $x.$

Next, we recall  from \cite {NguyenVietAnh17} two conversion rules  for changing  specializations in the  same leaf. 
For this  purpose  let $y\in L_x$ and pick $\xi\in \phi_x^{-1}(y).$
Since  the holonomy cocycle is  multiplicative (see \cite[eq. (2.11)]{NguyenVietAnh18b}), the  first  conversion rule (see \cite[identity (9.6)] {NguyenVietAnh17}) states that 
\begin{equation}\label{e:change_spec}
\kappa_y\big({\zeta-\xi\over  1-\zeta\bar\xi}\big)=\kappa_x(\zeta)-\kappa_x(\xi),\qquad \zeta\in\D.
\end{equation}
Consequently, since $\Delta_P$ is invariant  with respect to the automorphisms of $\D,$  it follows that
\begin{equation}\label{e:change_spec_bis}
\Delta_P \kappa_y(0)=\Delta_P\kappa_x(\xi).
\end{equation}
  
By   \cite[identities (9.5) and  (9.8)]{NguyenVietAnh17}, we have that
\begin{equation}\label{e:varphi_n_diffusion}
\kappa_x(0)=0  \quad\textrm{and}\quad  \E_x[\log {\|\mathcal H(\bullet,t)\|    }  ] =  
 (D_t \kappa_x)(0),\qquad t\in\R^+,
\end{equation}
where  $(D_t)_{t\in\R^+}$ is the family of  diffusion  operators associated with  $(\D,g_P).$

Consider the function $\kappa:\ X\setminus E\to\R$  defined by
\begin{equation}\label{e:kappa}
\kappa(x):= (\Delta_P\kappa_x)(0)\quad\text{for}\quad x\in X\setminus E.
\end{equation}

  \begin{remark}\rm  \label{R:well-defined} Let $\U$ be  a (regular or  singular) flow box.
  We  infer from \eqref{e:varphi_U}, \eqref{e:specialization} and \eqref{e:change_spec} that
  $|d\varphi_\U|_P,$ $\Delta_P \varphi_\U$ are globally well-defined on $X\setminus E,$ that is, they do not depend on the choice of $\U.$
  Moreover, we deduce from  \eqref{e:specialization},  \eqref{e:change_spec_bis} and \eqref{e:kappa} that
  \begin{equation}\label{e:curvature-varphi}
   \kappa(x)=-(\Delta_P\varphi_\U)(x)\qquad\text{for}\quad  x\in \U. 
  \end{equation}
  \end{remark}

Now let $L= \Nor(\Fc).$
Suppose that $L$ is  trivial over a flow box $\U\simeq \B\times \Sigma,$ i.e.
$L|_\U\simeq \U\times\C\simeq \B\times\Sigma\times\C.$
Consider the holomorphic section $e_L$ on $\U$ defined by $e_L(x):=(x,1).$
Let $\varphi$ be the local weight of $(L,g^\perp_X)$ with respect to $e_L.$
 Equality \eqref{e:hol_cocycle_dist} is  rewritten as follows
 \begin{equation*}
 {\|e_L(y)\|_{g^\perp_X}\over \|e_L(x)\|_{g^\perp_X}    }  ={e^{-\varphi(y)}\over  e^{-\varphi(x)}} ,
  \end{equation*}
where $x$  and $y$ are  on the same plaque  in the    flow box $\U\simeq \B\times \Sigma.$
Consequently,
\begin{equation}\label{e:hol_cocycle_dist_bis}
\ddcy \log\|\mathcal H(\omega,t)\||_{L_x}=\Big(\ddcy \log\big({\|e_L(y)\|_{g^\perp_X}\over \|e_L(x)\|_{g^\perp_X}    }\big)\Big)|_{L_x}= -\ddcy \varphi(y)|_{L_x}.
\end{equation}
Combining   \eqref{e:specialization},  
   \eqref{e:change_spec}, \eqref{e:change_spec_bis}, \eqref{e:kappa} and \eqref{e:hol_cocycle_dist_bis}, we obtain that
 \begin{equation}\label{e:kappa-varphi}
 \kappa(x)=-(\Delta_P\varphi)(x)\qquad\text{for}\qquad x\in X\setminus E.
 \end{equation}
 That  is why we call the function $\kappa$ the {\it  curvature density} of $\Nor(\Fc).$

\section{Preliminary results}\label{S:preliminaries}

\subsection{Transversal metric and holonomy variations near  singularities}
 
Fix an ambient Hermitian  metric $g_X$ on $X.$  
Fix  a  smooth Hermitian metric $g_0$ on the normal bundle $\Nor(\Fc)$ of $\Fc.$ Let $g^\perp_X$ be the metric on $\Nor(\Fc)$   defined  by \eqref{e:tran_metric}.
So  there is  a global weight function $f:\ X\to [-\infty,\infty)$ such that $g_X^\perp=g_0\exp{(-2f)}.$
We know that the weight function $f$ is   smooth outside $E.$ 

\begin{lemma}\label{L:trans_metric_near_sing}
\begin{enumerate}
\item If $g'_X$ is  another smooth Hermitian  metric on $X$ and  $g'^\perp_X$ is  the transversal metric associated  with $\Fc$ and $g'_X$ (see \eqref{e:tran_metric}), 
then there is  a constant $c>1$ such that
 $c^{-1}g^\perp_X\leq g'^\perp_X\leq cg^\perp_X$ on $X.$ 
 \item There is a constant $c>0$ such that
\begin{equation*}
 \omega^\perp_X\leq c g_X(x)\qquad\text{for}\qquad  x\in X\setminus E,
 \end{equation*}
 where $\omega^\perp_X$ is the transversal form given in \eqref{e:tran_form}.
\item Suppose  that $g_X$ coincides with the Euclidean  metric in a local model near a singular point $a$ of $\Fc.$ Then  for  $x=(z,w)\in \U_a\simeq \D^2,$
\begin{equation}\label{e:f_smoothness}f(x)={1\over 2}\log{(|z|^2+|\lambda w|^2 ) }+\text{a smooth function in}\ x.
\end{equation}
\item There is a constant $c>0$  such that
$$
|df(x)| \leq c (\dist(x,E))^{-1}g_X(x)\quad\text{and}\quad  |\ddc f(x)|\leq c (\dist(x,E))^{-2}g_X(x)\quad\text{for}\quad  x\in X\setminus E.$$
\end{enumerate}
Here, in assertions (2) and (4),  $g_X$ also denotes  the fundamental form associated to the 
Hermitian meric $g_X.$
 
\end{lemma}
\proof
 
\noindent{\bf Proof of assertion (1).} Since    there is a constant $c'>1$ such that $c'^{-1}g_X\leq g'_X\leq cg_X$ on $X,$ assertion (i)   follows  immediately   from
\eqref{e:tran_metric}.

\noindent{\bf Proof of assertion (2).} Consider  first   the case  where $g_X$ coincides with the Euclidean  metric in a local model near a singular point $a$ of $\Fc.$ 
Since  the vector $(z,\lambda w)$ is  tangent to  the leaf $L_x$ at $x=(z,w),$ the unit vector field 
$$N_x:={1\over\sqrt{|z|^2+|\lambda w|^2} }\big(-\bar\lambda\bar w{\partial\over \partial z}+\bar z{\partial\over \partial w}\big)$$ is  normal to $L_x$ at $x.$
Using this and  \eqref{e:tran_form}, we get assertion  (2)  in this case.

The  general  case is  similar  using assertion (1).

\noindent{\bf Proof of assertion (3).}
Consider  the local holomorphic  section $e_L$ given by  $(z,w)\mapsto {\partial\over\partial z}$ of $\Tan(X)$  over  $\U_a\simeq\D^2.$  
This  section induces  a holomorphic  section $\tilde e_L(x)=e_L(x)/\Tan(\Fc)_x$ of $\Nor(\Fc)$ over $\U_a.$
We  have, for $x=(z,w)\in\D\times (\D\setminus\{0\}),$
$$
\exp(-\varphi(x))=|\tilde e_L(x)|_{g^\perp_X}={1\over\sqrt{|z|^2+|\lambda w|^2} }|\lambda w|.
$$
Hence,   for $x=(z,w)\in\D^2\setminus\{(0,0)\},$
\begin{equation}\label{e:curvature_Nor}c_1(\Nor(\Fc), g^\perp_X)(x)=\ddc\varphi(x)=\ddczw\log{\sqrt{|z|^2+|\lambda w|^2}}.
\end{equation}
 Moreover,  in the local model with coordinates $(z,w)$ associated to the singular point
$E\ni a\simeq (0,0)\in\D^2,$ it follows  from \eqref{e:curvature_Nor} that
\begin{equation*}
c_1(\Nor(\Fc),g_0)+\ddc f= c_1(\Nor(\Fc), g^\perp_X)(x)=\ddc\varphi(x)={1\over 2}\ddczw\log{(|z|^2+|\lambda w|^2)}.
\end{equation*}
Consequently, assertion (3)  follows.

\noindent{\bf Proof of assertion (4).}
We only need to prove  this  assertion in a local model for a singular point $a\in E,$ that is,
to prove the lemma  for  $x\in \U_a\simeq \D^2.$ 
If 
$g_X|_{\U_a}$ coincides with the standard Euclidean metric on $\D^2,$
assertion (4)  is an immediate consequence of assertion (3).
For  the  general case  we argue as in the  proof of Lemma  \ref{L:laplacian_hol} (2) below.
\endproof

 The  following result gives  precise variations up to order $2$ of   $\mathcal H$ near a singular point $a$ 
 using the local model  $(\D^2,\Lc,\{0\})$ introduced in  Subsection   \ref{SS:local_model}. 
\begin{lemma}\label{L:laplacian_hol}  \begin{enumerate}
\item
Let $\D^2$ be  endowed with the Euclidean metric. Then
there is a constant $c>1$  such that
for every  $x=(z,w)\in({1\over2}\D)^2,$  we have that  
\begin{eqnarray*}
c^{-1}    \lof {\|(z,w)\|}&\leq& |d \kappa_x(0)|_P\leq c\lof {\|(z,w)\|},\\
-c  {|z|^2|w|^2\over (|z|^2+|w|^2)^2}  (\lof {\|(z,w)\|})^2&\leq& \Delta_P \kappa_x(0)\leq -c^{-1} {|z|^2|w|^2\over (|z|^2+|w|^2)^2}  (\lof {\|(z,w)\|})^2,
\end{eqnarray*}
where the function $\kappa_x$ is defined in \eqref{e:specialization}.
\item
Let $\D^2$ be  endowed with a smooth Hermitian metric. Then
there is a constant $c>1$  such that
for every  $x=(z,w)\in({1\over2}\D)^2,$  we have that  
\begin{eqnarray*}
c^{-1}    \lof {\|(z,w)\|}&\leq& |d \kappa_x(0)|_P\leq c\lof {\|(z,w)\|},\\
| \Delta_P \kappa_x(0)| &\leq & c {|z|^2|w|^2\over (|z|^2+|w|^2)^2}  (\lof {\|(z,w)\|})^2 +c \sqrt{|z|^2+|w|^2}(\lof {\|(z,w)\|})^2.
\end{eqnarray*}
\end{enumerate}
\end{lemma}
\proof  \noindent {\bf Proof of assertion (1).}
Let $\Phi_x$ be the function defined in Lemma \ref{L:holonomy}. For $\zeta\in \C$  with $|\zeta|\ll 1$  let $\xi(\zeta)\in\Pi_x$ be such that
$$\psi_x(\xi(\zeta))=\phi_x(\zeta).$$
Differentiating both sides    and using  \eqref{e:extremum} yield that
$$|\xi'(0)|=    {|d\phi_x(0)|\over |d\psi_x(0)|}       = {\eta(x)\over |d\psi_x(0)|}.$$
This, Lemma  \ref{L:holonomy} and \eqref{e:specialization} together  imply that
\begin{equation}\label{e:change_eta}
 |d \kappa_x(0)|=  \big|{d\log \|\mathcal H(\phi_x\circ \omega,1)\|\over d\zeta}|_{\zeta=0}\big|=
 \big|{d\log \Phi_x(\xi(\zeta))\over d\zeta}|_{\zeta=0}\big|
 =
 {\eta(x) |d\log{\Phi_x}(0)|\over |(d\psi_x)(0)|} .
\end{equation}
Using  \eqref{e:relation_Poincare_Hermitian_metrics}, a similar argument  shows that
\begin{equation}\label{e:change_eta_bis}
  (\Delta_P \kappa_x)(0)={\eta^2(x) |\ddc\log{\Phi_x}(0)|\over |(d\psi_x)(0)|^2}.
\end{equation}
We infer from \eqref{e:leaf_equation} that
\begin{equation}\label{e:dpsi_x}
|(d\psi_x)(0)|\approx \|(z,w)\|=\dist(x,E).
\end{equation}
 Next, observe  that 
 $
\dbar\log{|e^{i\zeta}|} + \dbar\log{|e^{i\lambda \zeta}|}=O( 1).$
 A  straightforward computation  shows that
\begin{equation}\label{e:dbar_log}
\dbar\log{ \big( |ze^{i\zeta}|^2+|\lambda we^{i\lambda \zeta}|^2     \big)}={\dbar |ze^{i\zeta}|^2+\dbar |\lambda we^{i\lambda \zeta}|^2   \over\big( |ze^{i\zeta}|^2+|\lambda we^{i\lambda \zeta}|^2
\big) }
= {-i\big(|ze^{i\zeta}|^2+ \bar\lambda|\lambda we^{i\lambda \zeta}|^2     \big)d\bar\zeta\over\big( |ze^{i\zeta}|^2+|\lambda we^{i\lambda \zeta}|^2     \big) }. 
\end{equation}
 Therefore, it follows that
 $$
 |d\log\Phi_x(0)|=\big|d \log{|e^{i\zeta}|} + d\log{|e^{i\lambda \zeta}|}- {1\over 2}d\log{ \big( |ze^{i\zeta}|^2+|\lambda we^{i\lambda \zeta}|^2     \big)}\big|=O(1).
 $$
Inserting this and  \eqref{e:dpsi_x} into  the  first equality in \eqref{e:change_eta} and applying  Lemma \ref{L:Poincare}, the first  inequality of assertion (1) follows.

 To  prove the second  inequality of assertion (1),
 note that
 $
i\ddbar\log{|e^{i\zeta}|} = i\ddbar\log{|e^{i\lambda \zeta}|}=0.$
Therefore,  we infer from Lemma \ref{L:holonomy}  that
$$
i\partial\dbar \log{\Phi_x(\zeta)}= -{1\over 2} i\ddbar\log{ \big( |ze^{i\zeta}|^2+|\lambda we^{i\lambda \zeta}|^2     \big)}.
$$
To  compute the right hand side of the last line, we use \eqref{e:dbar_log}
\begin{multline*}
\ddbar\log{ \big( |ze^{i\zeta}|^2+|\lambda we^{i\lambda \zeta}|^2     \big)}
= {-i\partial \big(|ze^{i\zeta}|^2+ \bar\lambda|\lambda we^{i\lambda \zeta}|^2     \big)\wedge d\bar\zeta\over\big( |ze^{i\zeta}|^2+|\lambda we^{i\lambda \zeta}|^2     \big) }\\
-{i\big(|ze^{i\zeta}|^2+ \bar\lambda|\lambda we^{i\lambda \zeta}|^2     \big)d\bar\zeta
\wedge \partial \big( |ze^{i\zeta}|^2+|\lambda we^{i\lambda \zeta}|^2     \big)
\over\big( |ze^{i\zeta}|^2+|\lambda we^{i\lambda \zeta}|^2     \big)^2 }\\
= {\big( |ze^{i\zeta}|^2+|\lambda|^2|\lambda we^{i\lambda \zeta}|^2     \big)\big( |ze^{i\zeta}|^2+|\lambda we^{i\lambda \zeta}|^2     \big)-
\big| |ze^{i\zeta}|^2+\lambda|\lambda we^{i\lambda \zeta}|^2     \big|^2  \over\big( |ze^{i\zeta}|^2+|\lambda we^{i\lambda \zeta}|^2     \big)^2 }d\zeta\wedge d\bar\zeta. 
\end{multline*}
We have shown that
$$
i\ddbar \log{\Phi_x(\zeta)}={-|\lambda-1|^2\over 2}{   |ze^{i\zeta}|^2|\lambda we^{i\lambda \zeta}|^2          \over \big( |ze^{i\zeta}|^2+|\lambda we^{i\lambda \zeta}|^2     \big)^2}id\zeta\wedge d\bar\zeta.
$$

Inserting this and  \eqref{e:dpsi_x} into  the  second equality in \eqref{e:change_eta} and applying  Lemma \ref{L:Poincare}, the second  inequality of assertion (1) follows.

\noindent {\bf Proof of assertion (2).} 
The first  inequality  can be proved  similarly  as in assertion (1). We  only prove   the second inequality.

First  consider the case where   $\D^2$ is  endowed with a  constant  Hermitian  metric $g_X,$ that is, $g_X(x)=g_X(0)$ for $x\in\D^2.$
This  case  is  analogous  to the  case of Euclidean metric  in assertion (1), although  the computation is  slightly more  involved.
We only give  a  sketchy proof here.
There are  $a_1,a_2,b_1,b_2\in\C$ such that  $a_1b_2\not=b_1a_2$ and that 
$$
g_X(0)(Z,W)= |a_1Z+b_1W|^2+|a_2Z+b_2W|^2\qquad\text{for}\qquad (Z,W)\in\C^2.
$$
For $x=(z,w)\in \D^2\setminus \{(0,0)\},$ $\kappa(x)$ is  determined by   
\begin{equation*}
 -\ddc\log{(  |a_1Z+b_1W|^2+|a_2Z+b_2W|^2 )}(0)|_{\widehat L_x}= \kappa(x)g_P(x),
\end{equation*}
where $(Z(\zeta),W(\zeta))=(ze^{i\zeta},we^{i\lambda\zeta}),$ $\zeta\in\C$ (see \eqref{e:leaf_equation}).
Write
$$
\pi\ddc\log{(  |a_1Z+b_1W|^2+|a_2Z+b_2W|^2 )}= { i\gamma\wedge\bar\gamma\over(  |a_1Z+b_1W|^2+|a_2Z+b_2W|^2 )^2},
$$
where $\gamma:=(  a_1Z+b_1W)(a_2dZ+b_2dW )-(  a_2Z+b_2W)(a_1dZ+b_1dW ).$
Using this and applying Lemma \ref{L:Poincare},
a straightforward  computation shows that $\kappa(x)$ satisfies the second inequality of assertion (2).

 Consider the  general case. 
Let $g'_X$ be  a Hermitian metric on $X$  such that  in a local model $\U_a\simeq \D^2$ near every singular point $a\in E$ it coincides with a constant  Hermitian metric  on $\D^2.$
 We see that there is  a  smooth  function  $\varphi$ on $\D^2\setminus \{(0,0)\}$ such that  $g_X^\perp=e^{-2\varphi}g'^\perp_X$  on $\D^2\setminus \{(0,0)\}.$
 Let $\kappa'_x$ be the function given in \eqref{e:specialization} using  $g'_X$ instead of $g_X.$
 By   \eqref{e:local_varphi}, \eqref{e:varphi_U} and  \eqref{e:specialization}, one has
 \begin{equation*}
  \kappa_x(\zeta)=\kappa'_x(\zeta)+(\varphi\circ\phi_x)(\zeta)\qquad\text{for}\qquad\zeta\in\D.
 \end{equation*}
 So we get 
  \begin{equation}\label{e:change_kappa}
  \Delta_P\kappa_x(0)=\Delta_P\kappa'_x(0)+(\Delta_P\varphi)(x)\qquad\text{for}\qquad x\in X\setminus E.
  \end{equation}
 Observe that $g_X(x)=g_X(0)+O(\|x\|)$ for $x\in({1\over 2}\D)^2.$
 Consequently,   $\varphi(x)=(O(\|x\|).$
 Hence, $\ddc\varphi=O(\|x\|^{-1}).$ Therefore, by  Lemma \ref{L:Poincare},
 $(\Delta_P\varphi)(x)=O(\|x\|(\lof{\|x\|})^2).$  Note  that  $\Delta_P\kappa'_x(0)$ satisfies the  second   estimate of assertion (1).
 Inserting this  in    \eqref{e:change_kappa}  and invoking the above estimate for $\Delta_P\kappa'_x(0),$   the second inequality of assertion (2)
   follows.

\endproof

\subsection{Lyapunov exponent and  weight function $W$}
 
 We  keep the hypotheses and notation of Theorem \ref{T:VA}.
 For $t\in\R^+,$ consider  the function $F_t:\  X\setminus E \to\R$ defined by
 \begin{equation}\label{e:F}
  F_t(x):= \int_{\Omega} \log \|\mathcal{H}(\omega,t)\| dW_x(\omega)  \qquad\text{for}\qquad x\in X\setminus E.
 \end{equation}
By  \eqref{e:Lyapunov_exp}
 the  Lyapunov exponent $\chi(T)$ can be rewritten as
\begin{equation}\label{e:Lyapunov_exp_bis}
\chi(T):= \int_X F_1(x)d\mu(x)={1\over t}\int_X F_t(x)d\mu(x),
\end{equation}
 where the  measure $\mu$ is  given in \eqref{e:mu}.
 The following result is  needed.
 \begin{theorem}\label{T:VA_bis} 
 There is a constant $c>0$ such that
  \begin{equation}\label{e:growth_F}
   |F_1(x)|\leq  c  \lof \dist(x,E)\qquad \text{for}\qquad  x\in X\setminus E.
 \end{equation}
 Moreover, the following inequality holds:
 \begin{equation}\label{e:integrability}
 \int_X \lof \dist(x,E)d\mu(x)<\infty.
\end{equation}
 \end{theorem}
\proof
\eqref{e:growth_F} follows from  Proposition 3.3 and Lemma 4.1 in \cite{NguyenVietAnh18b}.

Theorem 1.4 in  \cite{NguyenVietAnh18b} gives \eqref{e:integrability}. 
\endproof

 Define a weight function $W:\ X\setminus E\to \R^+$ as follows. Let $x\in X\setminus E.$
 If $x$ belongs to  a regular flow box then $W(x):=1.$
 Otherwise, if  $x=(z,w)$ belongs to a singular flow box $\U_a\simeq (\D^2,\Lc,\{0\}),$ $a\in E,$  which is  identified with the local model with coordinates $(z,w),$ then
 \begin{equation}\label{e:W}
  W(x):=   \lof{\|(z,w)\|}+{|z|^2|w|^2\over (|z|^2+|w|^2)^2}  (\lof {\|(z,w)\|})^2.
 \end{equation}
Note that $$1\leq \lof \dist(x,E) \leq W(x)\leq 2 (\lof \dist(x,E))^2.$$

 \section{Cohomological formula for the Lyapunov exponent}\label{S:Lyapunov}
  
   In this  section we first  prove Proposition   \ref{P:Lyapunov_vs_kappa} which plays a  crucial role in this  article. Next,  we  prove the first identity of Theorem A.
Inspired by  Definition 8.3  in Candel \cite{Candel03},   the  following notion is  needed.
\begin{definition}\label{D:moderate_function} \rm A real-valued function $h$ defined on   $\D$
 is called    {\it weakly moderate} 
  if there is a   constant $c>0$ such that
$$
\log |h(\xi)-h(0)|\leq  c\,\dist_P(\xi,0)+c,\qquad  \xi\in \D.
$$  
  \end{definition}
  \begin{remark}
   \rm
   The notion of  weak moderateness is  weaker than the notion of moderateness
   given   in \cite{Candel03,NguyenVietAnh18d}.
  \end{remark}

 The usefulness of weakly moderate functions is illustrated by 
 the following Dynkin type formula.
\begin{lemma}\label{L:D_t_Delta} 
Let $f\in\Cc^2(\D)$  be such that  $f,$ $|df|_P$ (see \eqref{e:df_P}) and $\Delta_P f$  (see \eqref{e:Delta_P})  are weakly moderate functions.   Then 
$$
(D_tf)(0) - f(0)=\int_0^t (D_s \Delta_P  f) (0) ds,\qquad t\in\R^+.
$$ 
\end{lemma}
\proof Lemma 5.1 in \cite{NguyenVietAnh18d} shows that
if $f,$ $|df|_P$  and $\Delta_P f$  are  moderate functions, then 
$$
(D_tf)(0) - f(0)=\int_0^t (D_s \Delta_P  f) (0) ds,\qquad t\in\R^+,\qquad\xi\in\D.
$$
Now under the  weaker  hypothesis that $f,$ $|df|_P$ and $\Delta_P f$  are weakly moderate,
 the same proof shows that  the above equality holds for $\xi=0.$ This proves the lemma.
\endproof
 The next lemma  shows us how  deep a leaf can go into  a singular flow  box before the hyperbolic  time $R.$
\begin{lemma}\label{L:go_deeper}{\rm  \cite[Lemma 3.2]{NguyenVietAnh18b}}
There is a constant $c>0$ with the following property. 
Let $x=(z,w)\in  (1/2 \D)^2$ and $\xi\in \D$   be such that
$\phi_x(t\xi)\in (1/2 \D)^2 $ for all $t\in[0,1].$  
Write $y:=\phi_x(\xi)$ and $R:=\dist_P(\xi,0).$
Then there  exists  $\zeta\in \Pi_x$  (see \eqref{e:leaf_equation} above) such that 
  $    y   =      (z e^{\zeta },we^{\lambda\zeta }) $
  and that
  $$|\zeta|\leq  e^{c R}|\log \|x\||.$$
 \end{lemma}

 The  following   result gives  an estimate on  the  expansion rate up to order $2$ of  $\mathcal H(\omega,\cdot)$ in terms of   $\dist_P(\cdot,0)$
  and  the  distance  $\dist(x,E).$  
\begin{proposition}\label{P:expansion_rate}
  There  is a  constant $c>0$ such that for every $x\in X\setminus E$ and every
  $ \xi\in\D,$
 \begin{eqnarray*}  \big |  \kappa_x(\xi) -\kappa_x(0)   \big |
 &\leq & c \lof \dist(x,E)\ \cdot\  \exp{\big (c\, \dist_P(\xi, 0)\big)}, \\
  \big |  |d\kappa_x(\xi) |_P-|d\kappa_x(0)|_P   \big |
& \leq& c \lof \dist(x,E)\ \cdot\  \exp{\big (c\, \dist_P(\xi,0)\big)}, \\
  \big |  \Delta_P\kappa_x(\xi )-\Delta_P\kappa_x(0)   \big |
 &\leq&  c (\lof \dist(x,E))^2\ \cdot\  \exp{\big (c\, \dist_P(\xi, 0)\big)}.
\end{eqnarray*}  
\end{proposition}
\begin{proof} Fix $\xi\in\D$ and   set $y:=\phi_x(\xi).$ 
Let $\Uc$ be the    finite  cover   of $X$  by  regular  and  singular  flow boxes  given in Subsection \ref {SS:local_model}.
We consider  three steps.\\
{\bf Step   1:}  {\it If  there is a singular  flow  box $\U$  which  contains  the whole segment $\{\phi_x(t\xi):\ t\in[0,1]\},$ then the proposition is  true for  $c=c_1,$
where $c_1>0$ is  a constant large enough.}

Write $x=(z,w)$ and 
  $R:= \dist_P(0,\xi).$ By Lemma \ref{L:go_deeper}, we may write
$y =(ze^\zeta,we^{\lambda\zeta})$ for some $\zeta\in\C$   such that
 \begin{equation}\label{e:zeta_R}|\zeta|\leq  e^{c_2R}.
 \end{equation}
 Inserting   this  into the  expression  for the holonomy map given in Lemma \ref{L:holonomy},
a straightforward computation shows  that
$$
  \big |  \kappa_x(\xi) -\kappa_x(0)   \big |=\big|\ln  \Phi_x(\zeta)\big|\leq  c_3  |\log \|x\||  e^{c_3R} 
$$  
  for a constant $c_3>0$ independent of $x$ and $y.$

  Next,  we deduce from the first  inequalities  of Lemma \ref{L:laplacian_hol} (2) 
  that
\begin{equation*}
\big |  |d\kappa_x(\xi) |_P-|d\kappa_x(0)|_P   \big | \leq  |d\kappa_x(\xi) |_P+|d\kappa_x(0)|_P   
\lesssim       \lof {\|x\|}+\lof {\|y\|}\leq  c_3  |\log \|x\||  e^{c_3R} ,
\end{equation*}
  for a constant $c_3>0$ independent of $x$ and $y,$
  where the last inequality holds by  \eqref{e:zeta_R}.

  By \eqref{e:change_spec_bis}, we have
$\Delta_P \kappa_y(0)=\Delta_P\kappa_x(\xi).$  Consequently, we deduce from the second  inequalities  of Lemma \ref{L:laplacian_hol} (2) and  \eqref{e:zeta_R}
  that
\begin{equation*}
   \big |  \Delta_P\kappa_x(\xi )-\Delta_P\kappa_x(0)   \big |\leq   \big |  \Delta_P\kappa_y(0 )\big |+\big|\Delta_P\kappa_x(0)   \big |
\lesssim       (\lof {\|x\|})^2+(\lof {\|y\|})^2\leq  c_3  (\log \|x\|)^2  e^{c_3R} ,
\end{equation*}
for a constant $c_3>0$ independent of $x$ and $y.$

  Choosing $c_1>c_3$ large enough, Step 1 follows from the above estimates.

\noindent
{\bf Step   2:}  {\it   If  the whole segment  $\{\phi_x(t\xi):\ t\in[0,1]\}$
is contained  in a single   regular  flow  box $\U\in\Uc,$  then 
$$ |  \kappa_x(t\xi)    |\leq c_4\quad\text{and}\quad 
   |  d\kappa_x(t\xi) |_P
 \leq  c_4 \quad\text{and}\quad
  \big |  \Delta_P\kappa_x(t\xi )  \big |
 \leq  c_4\quad\text{
for all}\quad  t\in[0,1].$$
Here $c_4>0$ is a constant independent of $x$ and $y.$ In particular, the proposition  is true in this  case for $c=c_1,$ where
 $c_1>0$ is  a constant large enough.}
 
 Observe  that the geodesic segment   $\{\phi_x(t\xi):\ t\in[0,1]\}$ is contained in the unique  plaque  of $\U$     which   passes through $x.$
 Moreover,  by Lemma \ref{L:Poincare}, $\eta\approx 1$ on  $\U.$
 This, combined   with  the  description of the  holonomy map  on  the regular  flow box $\U,$    implies the above three estimates.
Therefore, choosing $c_1>c_4$ large  enough,   we have that 
$c_1 \lof \dist(x,E)\geq c_4.$
This  proves the proposition  in  Step  2.
  
 \noindent
{\bf Step  3:}  {\it  Proof of the proposition in the general case.}

 We only prove the last   inequality of the proposition. The other  two inequalities can be proved similarly. 
Consider the family  of all   finite subdivisions of 
 $[0,1]$ into  intervals $[t_{j-1},t_j]$  with  $1\leq j\leq n$  such that
 $t_0=0,$ $t_n=1$ and that   each segment  $\{\phi_x(t\xi):\ t\in [t_{j-1},t_j]\}$ is  contained in a single (regular or singular) flow box $\U_j$ for each $j$.
 Fix a  member of this  family  such  that the  number $n$ is  smallest possible.
 We may assume  without loss of generality that $n>1$ since     the  case  $n=1$ follows  either from  Step 1 (if
 $\U_1$  is  singular)  or from Step 2 (if $\U_1$ is  regular).
The minimality of $n$ implies that all $\phi_x(t_1\xi),\ldots,\phi_x(t_{n-1}\xi)$ belong to the union of all regular  flow boxes of $\Uc.$
Therefore,  there is a constant $r_0>0$ independent of $x$ and $y$ such that 
$$
  \dist_P(t_j\xi,t_{j+1}\xi)\geq  r_0,\qquad 1\leq j\leq n-1. 
$$
 Thus  
\begin{equation}\label{eq_n}
n\leq  1+   r_0^{-1} \dist_P(\xi,0)= 1+r_0^{-1} R.
\end{equation}
Moreover,
 there is a constant $c_5>1$ independent of $\omega$ such that 
$$
1\leq \big(\lof \dist(\phi_x(t_j\xi),E)\big)^2\leq  c_5,\qquad 1\leq j\leq n-1. 
$$
 Using this and  applying Step 1 to each singular box in the family   $(\U_j)_{j=1}^n$ and  applying Step 2 to each regular 
 flow  box in  the above family, we obtain that
 \begin{eqnarray*}
\big |  \Delta_P\kappa_x(t_1\xi )  -   \Delta_P\kappa_x(0 ) \big|
 &\leq & c_1 \big(\lof \dist(x,E)\big)^2\cdot\exp{\big (c_1 \dist_P(0,t_1\xi)\big)} ,\\
 \big |  \Delta_P\kappa_x(t_j\xi )  -   \Delta_P\kappa_x(t_{j-1}\xi) \big|
&\leq&   c_1c_5\exp{\big (c_1\dist_P(t_{j-1}\xi,t_j\xi)\big)}   ,\quad 2\leq j\leq n.
\end{eqnarray*}
 Summing  up the above  estimates,  we  get that
 \begin{multline*}
 \sum_{j=1}^n     \big |  \Delta_P\kappa_x(t_j\xi )  -   \Delta_P\kappa_x(t_{j-1}\xi) \big|
 \leq  c_1 \big(\lof \dist(x,E)\big)^2\cdot\exp{\big ( c_1 \dist_P(0,t_1\xi)\big)}\\
+\sum_{j=2}^n c_1c_5\exp{\big (c_1 \dist_P(t_{j-1}\xi,t_j\xi)\big)}.
\end{multline*}
 On the other hand,  we have that  
 $$
  \big |  \Delta_P\kappa_x(\xi )  -   \Delta_P\kappa_x(0) \big|\leq \sum_{j=1}^n     \big |  \Delta_P\kappa_x(t_j\xi )  -   \Delta_P\kappa_x(t_{j-1}\xi) \big|.
$$
 This,  coupled   with the previous estimate, gives that  
  \begin{equation}\label{eq_two_sums_new}
  \begin{split}
 \big |  \Delta_P\kappa_x(\xi )  -   \Delta_P\kappa_x(0) \big|&\leq c_1\big(\lof \dist(\omega(x,E)\big)^2\cdot\exp{\big ( c_0 \dist_P(0,t_1\xi)\Big)}\\
&+
\sum_{j=2}^n c_1c_5\exp{\big (c_1 \dist_P(t_{j-1}\xi,t_j\xi)\big)}.
\end{split}
\end{equation}
  Since    $
 \lof \dist(x,E)\geq 1$ for all $x\in X\setminus E,$   the right hand side of the last line is  dominated by
a constant times $\lof \dist(x,E)$ times
$$
\sum_{j=1}^n  \exp{\big( c_1 \dist_P(t_{j-1}\xi,t_j\xi)\big)}\leq   n\cdot \exp{\big (c_1 \dist_P(0,\xi)\big)},
$$
where the last inequality holds because  of the identity
 $$
  \dist_P(0,\xi)=\sum_{j=1}^n\dist_P(t_{j-1}\xi,t_j\xi).
  $$
  Inserting (\ref{eq_n}) into the right hand side  of the last inequality and  choosing $c>c_1$ large enough, 
we find that its left hand  side  is bounded by   $c\, \exp{ (cR)}.$
So  the  right hand  side   of (\ref{eq_two_sums_new}) is  also bounded by a constant times $\big(\lof \dist(x,E)\big)^2\cdot \exp{\big (c \,\dist_P(0,\xi)\big)},$ and the proof is  thereby completed.
\end{proof}
The  following  result  relates the Lyapunov exponent $\chi(T)$ to the function $\kappa$ defined in \eqref{e:kappa}. It plays  the key role in this  article. 
\begin{proposition}\label{P:Lyapunov_vs_kappa}
 Under the  hypotheses and notations  of Theorem \ref{T:VA},
  the integrals
 $\int_X|\kappa(x)|d\mu(x)$  and $\int_X W(x)d\mu(x)$  are  bounded,  and the following identity holds
 $$
 \chi(T)=\int_X\kappa(x)d\mu(x).
 $$ 
\end{proposition}

 The novelty of this  proposition  is  that the weight $W(x)$ behaves like $(\lof \dist(x,E))^2$   when  $x=(z,w)$ satisfies $|z|\approx|w|,$  whereas our previous  work \cite{NguyenVietAnh18b} (see 
 Theorem \ref{T:VA_bis} above)     
  only provides the $\mu$-integrability  of  the less  singular  weight $\lof \dist(x,E).$
\proof  We divide the proof into 2 steps.

\noindent{\bf Step 1:} {\it Assume in addition that  the ambient metric $g_X$  is  equal to the Euclidean metric in a local model near  every  singular point of $\Fc.$}
By Proposition \ref{P:expansion_rate}, $\kappa_x,$
$|d\kappa_x|_P$ and $\Delta_P\kappa_x$ are weakly moderate functions  on $\D.$
Consequently, applying Lemma \ref{L:D_t_Delta} yields that
\begin{equation}\label{e:D_t_Delta}
(D_1\kappa_x)(0) - \kappa_x(0)=\int_0^1 (D_s (\Delta_P \kappa_x) ) (0) ds.
\end{equation}
By \eqref{e:varphi_n_diffusion}  and \eqref{e:F}, the left-hand side of \eqref{e:D_t_Delta} is 
 equal to
$$\E_x[\log \mathcal H(\omega,1)]=F_1(x),$$  
which is  finite   because of \eqref{e:growth_F}.
On the other hand, by \eqref{e:change_spec} and \eqref{e:kappa}, the right-hand side  of \eqref{e:D_t_Delta} can be rewritten as
$$
\int_0^1 (D_s  \kappa) (x) ds.
$$
Consequently, integrating  both sides of   \eqref{e:D_t_Delta}  with respect to  $\mu,$ we get that
\begin{equation}\label{e:F_t_vs_kappa}
\int_X  F_1(x)d\mu(x)=\int_X\big(\int_0^1 (D_s \kappa)(x)  ds\big)d\mu(x).
\end{equation}
Since  we know by  \eqref{e:Lyapunov_exp_bis}, \eqref{e:growth_F} and \eqref{e:integrability}  that the left intergral is   bounded and is  equal to $\chi(T),$ it follows that  right-side  double integral is also bounded. 

On the one hand, by the second inequalities in Lemma  \ref{L:laplacian_hol} (1) and \eqref{e:kappa}, $\kappa (x)\leq  0$  for every $x$   in a  singular flow box of  a  singular point  $a\in E.$
On the other hand,  using regular flow  box we see easily that $\kappa(x)\leq  c_0$  for every  $x$ outside the union of all singular flow boxes, where $c_0>0$ is  a constant.
Therefore,
\begin{equation}\label{e:kappa_upper_bound}
\kappa(x)\leq c_0\qquad\text{for}\qquad x\in X\setminus E.
\end{equation}
Moreover, $D_s$  is a positive  contraction for $s\in\R^+$ (see the second identity in \eqref{e:semi_group}.
Consequently, by Fubini's theorem,  we infer that for  almost every $s\in[0,1]$ with respect to the Lebesgue measure,
$$
\int_X (D_s \kappa) (x)d\mu(x)
$$
is bounded. 
Consequently,    by \eqref{e:kappa_upper_bound} and \eqref{e:semi_group} and  the obvious inequality $|\kappa|\leq 2c_0-\kappa,$
$$
\int_X(D_s|\kappa|)(x)d\mu(x)\leq  2c_0\int_X(D_s\textbf{1})(x)d\mu(x)  -\int_X(D_s\kappa)(x)d\mu(x)<\infty.
$$
Hence,  by  Proposition \ref{P:harmonic_currents_vs_measures} (4),
we infer that
$$
\int_X |\kappa (x)|d\mu(x)=\int_X (D_s |\kappa|) (x)d\mu(x)<\infty.
$$
 Thus,
we obtain that
$$
\int_X \kappa (x)d\mu(x)=\int_X (D_s \kappa) (x)d\mu(x).
$$
Inserting this  in  \eqref{e:F_t_vs_kappa} and using Fubini's theorem, we get that
$$
\chi(T)= \int_0^1\big(\int_X\kappa(x)  d\mu(x)\big)ds= 
\int_X \kappa (x)d\mu(x).
$$
Finally, using \eqref{e:W} and the  second  inequalities of Lemma \ref{L:laplacian_hol} (1) and inequality \eqref{e:integrability}, 
we  infer that
\begin{equation}\label{e:inequa_W}
\int_X W(x)d\mu(x)\lesssim \int_X |\kappa(x)|d\mu(x)+\int_X \lof{\dist(x,E)}d\mu(x)<\infty.
\end{equation}
The proof of Step 1 is thereby completed.

 \noindent{\bf Step 2:} {\it The general case.}  By \eqref{e:F_t_vs_kappa}, we have that 
 \begin{equation*}
\int_X  F_1(x)d\mu(x)=\int_X\big(\int_0^1 (D_s \kappa)(x)  ds\big)d\mu(x),
\end{equation*}
By the second inequalities of Lemma \ref{L:laplacian_hol} (2), there is a constant $c>0$ such that   $|\kappa(x)|\leq c  W(x)$
for $x\in X\setminus E.$ Hence,
$$
\int_X |\kappa(x)|d\mu(x)\leq c\int_X  W(x)d\mu(x).
$$ 
By \eqref{e:inequa_W}, the right  hand-side is  bounded. We infer  that
the  left hand-side  is also  bounded.
Hence,  we conclude the proof as in Step 1.
\endproof
\begin{corollary}\label{C:Lyapunov_vs_kappa}
 Under the  hypotheses and notations  of Theorem \ref{T:VA},   the following identities hold
 $$
  \kappa g_P=-c_1(\Nor(\Fc),g^\perp_X)\quad\text{on}\quad X\setminus E,\qquad\text{and}\qquad \int_X\kappa(x)d\mu(x)=-\int_X c_1(\Nor(\Fc),g^\perp_X)\wedge T.
 $$
\end{corollary}
\proof
 By \eqref{e:kappa-varphi}  we obtain that
 $$
 \kappa g_P=-(\Delta_P\varphi) g_P= -\ddc \varphi =-c_1(\Nor(\Fc),g^\perp_X)\qquad\text{on}\qquad \U.
 $$
 The first identity follows.  
Since 
 $\int_X|\kappa(x)|d\mu(x)<\infty$ by Proposition \ref{P:Lyapunov_vs_kappa},
  Integrating both sides of the first identity  over $X\setminus  E$ gives the second identity.  
\endproof
\proof[Proof of the first identity of   Theorem A]
As  in the proof of Proposition \ref{P:Lyapunov_vs_kappa}  we divide the proof into 2 steps.

\noindent{\bf Step 1:} {\it Assume in addition that  the ambient metric $g_X$  is  equal to the Euclidean metric in a local model near  every  singular point of $\Fc.$}
 
Fix  a  smooth Hermitian metric $g_0$ on the normal bundle $\Nor(\Fc)$ of $\Fc.$
So  there is  a global weight function $f:\ X\to [-\infty,\infty)$ such that $g_X^\perp=g_0\exp{(-2f)}.$
We know that the weight function $f$ is   smooth outside $E.$ 
Using  a finite partition of the unity on $X$ and applying   Lemma \ref{L:trans_metric_near_sing}  (see \eqref{e:f_smoothness}),  we can construct  a  family  of smooth  functions  $(f_\epsilon)_{0<\epsilon\ll 1}$ on $X$  such that
$f_\epsilon$ converges uniformly to $f$ in $\Cc^2$-norm on  each regular flow box as $\epsilon\to 0$ and that
 in a local model with coordinates $(z,w)$ associated to each singular point
$a\in E,$  
\begin{equation}\label{e:f_epsilon}f_\epsilon-{1\over2}\log{(|z|^2+|\lambda w|^2  +\epsilon^2)}=f-{1\over2}\log{(|z|^2+|\lambda w|^2 )}\qquad\text{on}\qquad \D^2.
\end{equation}
For every $0<\epsilon\ll 1$ we endow  $\Nor(\Fc)$  with  the  metric $g_\epsilon:=g_0\exp{(-2f_\epsilon)}.$
Since $g_\epsilon$ is  smooth and the current $T$ is  $\ddc$-closed, it follows that
\begin{equation}\label{e:cohomo_g_eps}
c_1(\Nor(\Fc))\smile \{T\}=\int_Xc_1(\Nor(\Fc),g_\epsilon)\wedge T.
\end{equation}
 Let $\kappa_\epsilon:\ X\setminus E\to\R$ be the function defined by 
 \begin{equation}\label{e:kappa_epsilon}
 -c_1(\Nor(\Fc),g_\epsilon)(x)|_{ L_x}= \kappa_\epsilon(x)g_P(x).
\end{equation}
 This, combined with \eqref{e:mu}, implies that
\begin{equation}\label{e:kappa_epsilon_bis}
 -c_1(\Nor(\Fc),g_\epsilon)\wedge T= \kappa_\epsilon d\mu.
\end{equation}
Since $f_\epsilon$ converges uniformly to $f$ in $\Cc^2$-norm on    compact subsets of $X\setminus E$  as $\epsilon\to 0,$ 
it follows that  $\kappa_\epsilon$ converge pointwise to $\kappa$ $\mu$-almost everywhere.
Hence, we  get that
\begin{equation*}
-\int_{X\setminus(\bigcup_{a\in E}  \U_a)}c_1(\Nor(\Fc),g_\epsilon)\wedge T=\int_{X\setminus(\bigcup_{a\in E}  \U_a)}\kappa_\epsilon(x)d\mu(x)
\to \int_{X\setminus(\bigcup_{a\in E}  \U_a)}\kappa(x) d\mu(x)\quad\text{as}\quad \epsilon\to 0.
\end{equation*}
We will  show that  on each  singular flow box $\U_a\simeq \D^2,$
\begin{equation}\label{e:inside}
-\int_{ \U_a}c_1(\Nor(\Fc),g_\epsilon)\wedge T\to \int_{  \U_a} \kappa(x)d\mu(x)\qquad\text{as}\qquad \epsilon\to 0.
\end{equation}
Taking   \eqref{e:inside} for granted, we combine it with  the previous limit and get that
$$
-\int_{X}c_1(\Nor(\Fc),g_\epsilon)\wedge T\to \int_X\kappa(x)d\mu(x)=-\int_{X}c_1(\Nor(\Fc),g^\perp_X)\wedge T  \qquad\text{as}\qquad \epsilon\to 0,
$$
where the last equality follows from  Corollary \ref{C:Lyapunov_vs_kappa}.
We deduce from this and \eqref{e:cohomo_g_eps} that
$$
-c_1(\Nor(\Fc))\smile \{T\}=\int_X\kappa(x)d\mu(x).
$$
By  Proposition \ref{P:Lyapunov_vs_kappa}, the right hand side is $\chi(T).$
Hence, the last   equality  implies the desired identity of the theorem.

Now  it remains to prove \eqref{e:inside}.
We need the  following result which gives a precise behaviour of $\kappa_\epsilon$ near a singular point $a$ 
 using the local model  $(\D^2,\Lc,\{0\})$ introduced in  Subsection   \ref{SS:local_model}. 
\begin{lemma}\label{L:laplacian_hol_bis}  
There is a constant $c>1$  such that for every $0<\epsilon\ll 1$ and 
for every  $x=(z,w)\in({1\over2}\D)^2,$  we have that  
\begin{multline*}
-c\big(  {|z|^2|w|^2\over (|z|^2+|w|^2+\epsilon^2)^2 }+  (|z|^2+|w|^2)\big)  (\lof {\|(z,w)\|})^2\leq  \kappa_\epsilon(x)\\
\leq \big(-c^{-1} {|z|^2|w|^2\over (|z|^2+|w|^2+\epsilon^2)^2} +c(|z|^2+|w|^2)\big)  (\lof {\|(z,w)\|})^2.
\end{multline*}  
\end{lemma}
\proof
Since  $g_\epsilon=g_0\exp{(-2f_\epsilon)}$ we get  that
$$
c_1(\Nor(\Fc),g_\epsilon)=c_1(\Nor(\Fc),g_0)+\ddc f_\epsilon=\ddc f_\epsilon+\text{a smooth $(1,1)$-form independent of $\epsilon$}.
$$
This,  together  with  \eqref{e:f_smoothness}, \eqref{e:f_epsilon} and \eqref{e:kappa_epsilon}, imply that
\begin{equation*}
 \kappa_\epsilon(x)g_P(x)= -{1\over 2} \ddc\log{(|z|^2+|\lambda w|^2+\epsilon^2)}(x)|_{ L_x}+\text{a smooth $(1,1)$-form independent of $\epsilon$}.
\end{equation*}
 Using the parametrization \eqref{e:leaf_equation} the pull-back of the first term of  the right-hand side  by $\psi_x$ is 
\begin{equation*}
 -{1\over 2}\ddc\log{(|ze^{i\zeta}|^2+|\lambda we^{i\lambda\zeta}|^2+\epsilon^2)}(0),
\end{equation*}
whereas  the pull-back of the second term of  the right-hand side  by $\psi_x$ is   $O(|z|^2+|w|^2) d\zeta\wedge d\bar\zeta.$ 
A straightforward computation as in the end of  the proof of Lemma \ref{L:laplacian_hol} shows that the former  expression is  equal to
\begin{equation*}
-{|\lambda-1|^2\over 4\pi}  {|z|^2|w|^2\over (|z|^2+|\lambda w|^2+\epsilon^2)^2} id\zeta\wedge d\bar\zeta.
\end{equation*}  
Using this and  \eqref{e:dpsi_x}, and applying Lemma \ref{L:Poincare}, the result follows.
\endproof

We resume the proof of  \eqref{e:inside}. By Lemmas \ref{L:laplacian_hol} and \ref{L:laplacian_hol_bis}, there is  a  constant $c>1$ such that for $(z,w)\in\U_a\simeq\D^2$ that
$$
|\kappa_\epsilon(z,w)|\leq c \big({|z|^2|w|^2\over (|z|^2+|w|^2+\epsilon^2)^2} +(|z|^2+|w|^2)\big) (\lof {\|(z,w)\|})^2\leq c^2 |\kappa(z,w)|+c^2.$$
Recall from Proposition \ref{P:Lyapunov_vs_kappa} that  $\int_{\U_a}|\kappa(x)|d\mu(x)<\infty.$ 
On the other hand,  $\kappa_\epsilon$ converge pointwise to $\kappa$ $\mu$-almost everywhere as $\epsilon\to 0.$
Consequently, by Lebesgue dominated convergence, 
$$
\lim\limits_{\epsilon\to 0}\int_{\U_a} \kappa_\epsilon d\mu  =\int_{\U_a} \kappa d\mu.
$$
This and  \eqref{e:kappa_epsilon} imply   \eqref{e:inside}.
The proof of Step 1 is thereby completed.

 \noindent{\bf Step 2:} {\it The general case} (two proofs).
 
 Consider a  general ambient  Hermitian  metric $\hat g_X$ on $X.$
 We keep  the notations introduced in Step 1.  Let  $\chi(T)$  (resp. $\hat\chi(T)$ be the Lyapunov exponent of $T$ when we  use the ambient metric  $g_X$ (resp. $\hat g_X$).
 We  only need to show  that     $\chi(T)=\hat\chi(T).$  There is  a constant $c>1$ such that
 $$c^{-1}g_X\leq  \hat g_X\leq c g_X.$$
 Consequently, we infer from \eqref{e:hol_cocycle_dist} that 
  $$   c^{-2}\| \mathcal H(\omega,t) \|_{g_X} \leq \| \mathcal H(\omega,t) \|_{\hat g_X}\leq  c^2 \| \mathcal H(\omega,t) \|_{g_X}$$  
 for    every $x\in X\setminus E$  and  every path  $\omega\in\Omega_x$  and every time  $t\in\R^+. $
 Hence,
 $$
 \lim\limits_{t\to \infty} \big({1\over  t} \log  \| \mathcal H(\omega,t) \|_{g_X} - {1\over  t} \log  \| \mathcal H(\omega,t) \|_{\hat g_X}\big) =0.
 $$
 By Theorem  \ref{T:VA} (2) applied to $\mu$-almost  every $x\in X\setminus E$  and to    almost every path  $\omega\in\Omega$ with respect to $W_x,$
 we get from  the  above  equality that $\chi(T)=\hat\chi(T).$

 Here is  an alternative proof of Step  2 which is of independent interest.
 By Proposition  \ref{P:Lyapunov_vs_kappa} and Corollary \ref{C:Lyapunov_vs_kappa} and , the  desired equality $\chi(T)=\hat\chi(T)$   amounts  to
 the  equality
 \begin{equation*} 
  \int_X c_1(\Nor(\Fc),g^\perp_X)\wedge T=\int_X c_1(\Nor(\Fc),\hat g^\perp_X)\wedge T.
 \end{equation*}
Consider the global weight function $\hat f:\ X\to [-\infty,\infty)$ satisfying  $\hat g_X^\perp=g_0\exp{(-2\hat f)}.$
So
\begin{equation}\label{e:f-hatf}  c_1(\Nor(\Fc), g^\perp_X)= c_1(\Nor(\Fc),g_0)+\ddc f\quad\text{and}\quad c_1(\Nor(\Fc),\hat g^\perp_X)= c_1(\Nor(\Fc),g_0)+\ddc \hat f.
\end{equation}
The  above  equality is  reduced  to 
\begin{equation}\label{e:Lyapunov_equal_1}
  \int_X  \ddc f\wedge T=\int_X \ddc \hat f\wedge T.
 \end{equation}
 Note  that  both  $f$ and $\hat f$ are smooth  outside $E.$
  Consider the  function $\tilde f:\ X\to [-\infty,\infty]$ given by  $\tilde f=f-\hat f.$
  Fix  $\epsilon_0>0$  small enough. There is  a  constant $c>1$ such that for every $0<\epsilon<\epsilon_0,$ there is   a  smooth  function $\theta_{\epsilon}:\ X\to [0,1]$
  such that  
   \begin{equation}\label{e:theta}\theta_{\epsilon} (x)=0\  \ \text{for} \ \ \dist(x,E)<\epsilon/2,\quad  \theta_{\epsilon} (x)=1 \ \  \text{for}\  \  \dist(x,E)>\epsilon\quad
   \text{and}\ \   |d\theta_{\epsilon}|\leq  c\epsilon^{-1},\quad
  |\ddc\theta_{\epsilon}|\leq  c\epsilon^{-2}.
  \end{equation}
  Fix  $0<\epsilon<\epsilon_0.$ Since $T$ is  $\ddc$-closed and both $f$ and $\hat f$ are smooth on a neighborhood of the support of $\theta_\epsilon$, we have that 
  $$
  \int_X  \ddc (\theta_\epsilon \tilde f)\wedge T=0.
  $$
   By Proposition  \ref{P:Lyapunov_vs_kappa} and Corollary \ref{C:Lyapunov_vs_kappa} and \eqref{e:f-hatf}, 
   we  get that
   $$
   \lim\limits_{\epsilon\to 0} \int_X  \theta_\epsilon \ddc\tilde f \wedge T=\hat\chi(T)-\chi(T).
   $$
   Therefore, equality \eqref{e:Lyapunov_equal_1} is  reduced to showing that each of the following  terms
  \begin{equation}\label{e:Lyapunov_equal_2}
  \int_X  d \tilde f\wedge \dc \theta_\epsilon\wedge  T,\quad  \int_X  \dc \tilde f\wedge d \theta_\epsilon\wedge  T,\quad \int_X  \hat f\ddc\theta_\epsilon \wedge T.
 \end{equation}
 tends to $0$ as $\epsilon$ tends to $0.$ 
 
 By Lemma \ref{L:trans_metric_near_sing} (1) and (4)  applied to $g_X$ and $\hat g_X,$  we get  a constant $c>0$ such that 
$$ |\tilde f(x)|\leq c \quad\text{and}\quad |  d\tilde f(x)| \leq c (\dist(x,E))^{-1}g_X(x)\quad\text{for}\quad  x\in X\setminus E.
$$
Using this and the  properties \eqref{e:theta} of the functions $\theta_\epsilon,$    \eqref{e:Lyapunov_equal_2} is  reduced to showing that
$$
\lim_{\epsilon\to 0}\epsilon^{-2} \int_{x\in X:\ \epsilon/2<\dist(x,a)<\epsilon} T\wedge g_X=0\qquad\text{for all}\qquad  a\in E.
$$
But this inequality is  equivalent to the  vanishing  of the Lelong number of $T$ at $a$ which has been established in \cite{NguyenVietAnh18a}.
Hence, the general case is completed.
\endproof

\begin{remark}\rm
There is an alternative  proof of Step 1 of  Proposition \ref{P:Lyapunov_vs_kappa}  which  is  based on the regularization as in the proof of the first identity of Theorem A
and a  cohomological argument.  This new method uses  the monotone convergence theorem instead of the Lebesgue dominated convergence. However,  it still relies on  Theorem \ref{T:VA_bis}.
 
\end{remark}

\section{Cohomological formula for the Poincar\'e mass}\label{S:Mass}
In this  section we   prove  the last identity  \eqref{e:coho_mass} of Theorem A.
As an application of this theorem, we  compute the Lyapunov exponent  of a generic
foliation   with degree $d>1$ in $\P^2$ (Corollary C).

First, we keep the hypotheses and notations  of Theorem  \ref{T:VA}.
 Let $\Fc$  be given by  an open covering $\{\U_j\}$ of $X$ and 
holomorphic vector fields $v_j\in H^0(\U_j,\Tan(X))$ with isolated singularities  satisfying  \eqref{e:Cotan} 
for some non-vanishing holomorphic functions $g_{jk}\in H^0(\U_j\cap U_k, \Oc^*_X).$
Recall   that  $\Cotan(\Fc)$  is  the  holomorphic line  bundle on $X$  associated  to   the  multiplicative cocycle  $g_{jk}.$  
Consider the singular Hermitian metric $h$ on $\Tan(\Fc)$  defined by
\begin{equation}\label{e:metric-h}
 h(v_j(x),v_j(x)):=  \|v_j(x)\|^2_P,
\end{equation}
where  on the  right-hand side  $\|v_j(x)\|_P$ denotes the  length of the tangent vector $v_j(x)$ measured with respect to  the leafwise Poincar\'e metric $g_P,$
that is, using \eqref{e:covering_map} and \eqref{e:extremum} and \eqref{e:tran_metric},
\begin{equation}\label{e:change-metrics}
  \|v_j(x)\|_P=|\phi^*_x(v_j(x))|={ \|v_j(x)\|_{g_X}\over \|d\phi_x(0)\|}={ \|v_j(x)\|_{g_X}\over \eta(x)}.
\end{equation}
Let $h^*$  be the  dual  metric of $h$ on  $\Cotan(\Fc).$
By Brunella \cite[Theorem 1.1]{Brunella03},  
the   curvature $c_1(\Cotan(\Fc),h^*)$ is a positive closed  current on $X.$ 
Fix  a  smooth  Hermitian metric $h_0$ on $\Tan(\Fc).$
Let $h_0^*$ be the  dual  metric of $h_0$  on  $\Cotan(\Fc).$
Hence, there is  a upper-semi continuous function $\psi:\ X\to [-\infty,\infty)$ such that $h^*=e^{-2\psi}h^*_0.$
So  the  above   result of Brunella  says that  
\begin{equation}\label{e:Brunella}
c_1(\Cotan(\Fc),h^*)=\ddc \psi+c_1(\Cotan(\Fc),h^*_0)\geq 0.
\end{equation}

The  following result is  needed.

\begin{lemma}\label{L:Cotan}
\begin{enumerate}
 \item 
We have 
  $
c_1(\Cotan(\Fc),h^*)(x)|_{L_x}=g_P(x)$  for  $x\in X\setminus E.$
\item $\psi$ is  smooth  function outside $E,$ and   
$$\psi(x)=\lof\lof \dist(x,E)+O(1)\qquad \text{for}\qquad  x\in X\setminus E.$$
\end{enumerate}
\end{lemma}
\proof
\noindent {\bf Proof of assertion (1).}
Let $\U=\U_j$ be  a  flow box containing $x.$  Let $0<r<1$  be so small such that $u$  is well-defined on $\phi_x(\D_r)\subset \U.$ 
Let $v:=v_j$  be the  vector  field associated with $\U.$
Let $\zeta\in\D_r$  and  write $y:=\phi_x(\zeta).$  Since $\phi_y$  is  the compose
of the disc-automorphism $\D\ni z\mapsto  {z+\zeta\over  1+z\bar\zeta}$   and $\phi_x,$ we infer  from  \eqref{e:change-metrics}
that
\begin{equation*}
  \|v(y)\|_P=|\phi^*_y(v(y))|={ \|v(y)\|_{g_X}\over \|d\phi_x(\zeta)\| (1-|\zeta|^2)}.
\end{equation*}
Observe that  as  $v$ is  a holomorphic vector field,  ${ \|v(\phi_x(\zeta))\|_{g_X}\over \|d\phi_x(\zeta)\|}$ is  the  modulus of a non vanishing holomorphic function.
Therefore,
$$
\ddc\log \|v(\phi_x(\zeta))\|_P=-\ddc \log (1-|\zeta|^2)=g_P(\zeta)\qquad\text{for}\qquad \zeta\in\D_r.
$$
Since the left-hand side is  equal to $$\phi_x^*\big(-c_1(\Tan(\Fc),h)\big)(\zeta)=\phi_x^*\big(c_1(\Cotan(\Fc),h^*)\big)(\zeta),$$
assertion (1) follows.

\noindent {\bf Proof of assertion (2).}
Let $g^\parallel_X$ be the metric on $\Tan(\Fc)$   induced    by $g_X.$
So  there is  a global weight function $f:\ X\to [-\infty,\infty)$ such that $g_X^\parallel=h_0\exp{(-2f)}.$
We know that the weight function $f$ is   smooth outside $E.$ 
 
 Suppose  without loss of generality that $g_X$ coincides with the Euclidean metric in a local model near every singular point $a$ of $\Fc.$
Consider  the local holomorphic  section $e_L$ given by  $(z,w)\mapsto z{\partial\over\partial z}+\lambda w{\partial\over\partial w}$ of $\Tan(\Fc)$  over  $\U_a\simeq\D^2.$  
We  have, for $x=(z,w)\in\D^2\setminus\{(0,0)\}),$
$$
\exp(-\varphi(x))=| e_L(x)|_{g^\parallel_X}=\sqrt{|z|^2+|\lambda w|^2} .
$$
Hence,   for $x=(z,w)\in\D^2\setminus\{(0,0)\},$
\begin{equation*}c_1(\Tan(\Fc), g^\parallel_X)(x)=\ddc\varphi(x)=-\ddczw\log{\sqrt{|z|^2+|\lambda w|^2}}.
\end{equation*}
 Moreover,  in the local model with coordinates $(z,w)$ associated to the singular point
$E\ni a\simeq (0,0)\in\D^2,$ we get from the last equalities that
\begin{equation*}
c_1(\Nor(\Fc),h_0)+\ddc f= c_1(\Tan(\Fc), g^\parallel_X)(x)=\ddc\varphi(x)=-{1\over 2}\ddczw\log{(|z|^2+|\lambda w|^2)}.
\end{equation*}
Consequently, it follows that
for  $x=(z,w)\in \U_a\simeq \D^2,$
\begin{equation*}f(x)=-{1\over 2}\log{(|z|^2+|\lambda w|^2 ) }+\text{a smooth function in}\ x.
\end{equation*}
So, there is a constant $c>1$ such that  for every local holomorphic  section  $v$ of $\Tan(\Fc)$ over an open set $\U,$  we have  
$$
c^{-1}\|v(x)\|_{g_X}\leq \|v(x)\|_{h_0} \dist(x,E)\leq \|v(x)\|_{g_X} \qquad \text{for}\qquad  x\in\U.
$$
On the other hand, we  deduce  from  \eqref{e:metric-h} and  \eqref{e:relation_Poincare_Hermitian_metrics} that
$$
\|v_j(x)\|_{g_X}=\|v_j(x)\|_P \eta(x)=\|v_j(x)\|_{h_0} e^{-\psi(x)}\eta(x)\quad\text{for}\quad x\in\U_j.
$$
This,  coupled with the last  inequalities, yields  that  
$$
 c^{-1}\dist(x,E)\leq e^{-\psi(x)}\eta(x)\leq c \dist(x,E)\quad\text{for}\quad x\in X.
$$ 
By Lemma \ref{L:Poincare},  $\eta(x)\approx  \dist(x,E) \lof \dist(x,E).$  
Hence,  we  get the  conclusion of assertion (2).
\endproof

We infer from Lemma  \ref{L:Cotan} and  \eqref{e:mu}  that
\begin{equation}\label{e:mu_bis}
  c_1(\Cotan(\Fc),h^*)\wedge T=\mu \qquad\text{on}\qquad X\setminus E.
\end{equation}


Multiplying  $ h^*_0$ by a positive  constant,
we may assume  without loss of generality that $\psi\leq -1.$
\begin{lemma}\label{L:BlockiKolodziej}
 There exist  a decreasing  sequence of positive real numbers  $(\epsilon_n)$  and a decresing sequence of smooth negative quasi-psh functions $\psi_n$   
with the following properties:
 $\epsilon_n\searrow 0,$ and  $\psi_n\searrow\psi$ on $X,$ and $\psi_n$ converge to $\psi$ locally uniformly on $X\setminus E$   as $n\to\infty,$ and  
\begin{equation*}
\ddc \psi_n+c_1(\Cotan(\Fc),h^*_0) +\epsilon_n  g_X\geq 0\quad\text{in the current sense on $X$, for every $n$}.
\end{equation*}
\end{lemma}
\proof
Lemma  \ref{L:Cotan}  says in particular that   the Lelong number of $\psi$ vanishes  everywhere. 
Using this
 we may perform B{\l}ocki--Ko{\l}odziej's regularization \cite[Theorem 2]{BlockiKolodziej} to obtain the conclusion of the lemma. However, since
 the unbounded locus\footnote{The unbounded locus of $\psi$ is  by definition the  set of  all points $a\in X$
such that $\psi$ is  unbounded on every neighborhood  of $a.$} of $\psi$ is  exactly  $E$  which is  set,
we prefer  to  give here   a simpler direct argument  for the  reader's convenience.

Fix  a finite cover $\Vc:=(\V_p)_{p\in I}$  of  (singular and  regular) flow boxes of $X$ such that $V_p\Subset \U_p,$ $p\in I,$
where  $\Uc:=(\U_p)_{p\in I}$ is  the cover  introduced in Subsection \ref{SS:local_model}. Fix an $\epsilon>0.$ We  choose  a smooth function $f_p$ in $\U_p$
such that
\begin{equation}\label{e:ddcf_p}
 0\leq \ddc f_p-c_1(\Cotan(\Fc),h^*_0)\leq \epsilon g_X \qquad\mathrm{in}\qquad \U_p.
\end{equation}
Then the function $\varphi_p:=\psi+f_p$ is plurisubharmonic in $\U_p.$
Every  flow box  $\U_p,$ $p\in I,$ can be  regarded as a bidisc $\D^2.$
we fix a regularization  by convolution for functions defined on $\U_p$ as  follows.
Fix a smooth radial function $\rho:\  \C^2\to  [0,\infty)$ with  compact support in $\D^2$   such that  $\int_{\C^2} \rho(x)\Leb(x)=1,$
where $\Leb$ denotes the Lebesgue measure of $\C^2.$  
For a  function $u:\  \U_p\simeq \D^2\to  [-\infty,\infty)$ and  a number $\delta>0,$ we consider the regularized function
$$
u_\delta(x):=(u\star \rho_\delta)(x)=\int_{\C^2}  u(x-\delta y)\rho(y) \Leb (y)\quad\text{for}\quad  x\in \U_p\simeq \D^2.
$$
Here, $\rho_\delta(x):=\delta^{-4} \rho(x/\delta),$ and on the right hand side the function $u$  is  extended to $\C^2$  by setting it equal to $0$ outside $\D^2.$
So  if $u$ is plurisubharmonic, then  so  is  $u_\delta$  and $u_\delta$ decrease to $u$ as $\delta\searrow 0.$  Moreover, if $u$ is continuous
then    $u_\delta$  converge  to $u$ locally  uniformly.

Since  for two different (regular or singular) flow boxes $\U_p,$ $\U_q\in\Uc,$ we always have
$(\U_p\cap \U_q)\cap E=\varnothing,$  it follows that
\begin{equation}\label{e:cv-uniform}
\varphi_{p,\delta}-\varphi_{q,\delta}
\quad\text{converge locally uniformly in}\quad \U_p\cap \U_q\quad \text{to}\quad f_p-f_q.
\end{equation}
Let $\chi_p$  be a  smooth function in $\U_p$   such that $\chi_p=0$ on $\V_p$
and  $\chi_p=-1$ away  from a  compact subset of $\U_p.$
So there is a constant $c>0$ such that  $\ddc \chi_p\geq -cg_X.$
For $\delta>0$ consider
$$
\varphi_\delta:=\max\limits_{p\in I}\big(\varphi_{p,\delta}-f_p+\epsilon c^{-1}\chi_p\big).
$$
Therefore, we deduce  from  \eqref{e:cv-uniform} that  for  $\delta>0$ small enough, the values on the set $\{\chi_p=-1\}$  do not contribute to
the  maximum. Hence, $\varphi_\delta$ is  continuous.  If  we consider regularized maximum  (see e.g. \cite{Demailly}) in place of maximum,  we obtain smooth function $\varphi_\delta.$
This, combined  with  \eqref{e:ddcf_p}, implies that
$$
\ddc\varphi_\delta+c_1(\Cotan(\Fc),h^*_0)+ 2\epsilon g_X\geq 0.
$$
Moreover, $\varphi_\delta\searrow \psi$ as $\delta\searrow 0.$
Now for each  $n,$  we  set $\epsilon:=\epsilon_n/2$ and $\psi_n:=\varphi_\delta$  with $\delta>0$  small  enough.
Since $\psi\leq 1,$ $\varphi_n$ can be chosen  to be negative. 
This completes the proof.
\endproof
By Lemma \ref{L:BlockiKolodziej} $\psi_n\leq 0$ and  $\psi_n\searrow\psi$ as $n$ tendsto infinity.
This, coupled with Lemma \ref{L:Cotan} (2), gives  a constant $c$ independent of $n$ such that
\begin{equation}\label{e:bound_psi_n} 
|\psi_n|\leq |\psi|\leq \lof\lof{\dist(x,E)}  +c.
 \end{equation}
For $n\geq 1$  define the  following measure on $X:$
\begin{equation}\label{e:nu_n}
 \nu_n:= \big(\ddc \psi_n+c_1(\Cotan(\Fc),h^*_0)+\epsilon_n  g_X\big)\wedge T.
\end{equation}
By Lemma \ref{L:BlockiKolodziej} $\nu_n$  are all  positive  finite.

\begin{lemma}\label{L:W_bull}
There is  a smooth function  $ W_\bullet:\  X\setminus E\to \R^+$ and  a constant $c>1$ such that
\begin{eqnarray*}
 c^{-1}\lof \lof \dist(x,E))&\leq & W_\bullet(x) \leq  c\lof \lof \dist(x,E)),\\
|d W_\bullet(x)|_{L_x}|&\leq&  c(\dist(x,E))^{-1}(\lof \dist(x,E))^{-1},\\
|\ddc W_\bullet(x)|_{L_x}|&\leq&  c(\dist(x,E))^{-2}(\lof \dist(x,E))^{-2}
 \end{eqnarray*}
 for  all $x\in X\setminus E.$  
\end{lemma}
\proof
Let $\rho:\ \R\to  [0,1]$ be  a  smooth increasing  function   such that  $\rho(t)=0$ for $t\leq 0$  and $\rho(t)=1$ for $t\geq 1.$
Consider the   ``regularized min'' function
$$
m(t,s):= t-\rho(t-s).
$$  
We check easily  that  $m$ is  smooth  and its  first  and second derivatives  are  uniformly bounded  
and 
$$ \min \{t,s\}-1\leq  m(t,s)\leq  \min \{t,s\}+1.$$
Consider   a function $ W_\bullet:\  X\setminus E\to \R^+$ which is  smooth  outside the singular flow boxes $\U_a$ for $a\in  E$ and  which  is  defined  
in a local model near a singular point  $a$  by
$$
 W_\bullet(x):=\log m(-\log|z|,-\log|w|   )\qquad\text{for}\qquad  x=(z,w)\in \D^2\simeq \U_a.
$$
Using   \eqref{e:leaf_equation_bis} and \eqref{e:|z|} for $x=\tilde\psi_\alpha(\zeta)$ with $\zeta=u+iv\in\C$ and $\alpha\in \A,$  we  get that
\begin{equation*}
W_\bullet(x)=\log m\big(v, (\Im \lambda) u+(\Re \lambda) v\big).
\end{equation*} 
Using  the above  properties of $m,$ a straighforward computation shows that the first derivative and the second one  with respect to $u,v$  satisfy
\begin{eqnarray*}
D\log m\big(v, (\Im \lambda) u+(\Re \lambda) v\big)&=&O\big((\min\{v, (\Im \lambda) u+(\Re \lambda) v\})^{-1}\big),\\
D^2\log m\big(v, (\Im \lambda) u+(\Re \lambda) v\big)&=&O\big((\min\{v, (\Im \lambda) u+(\Re \lambda) v\})^{-2}\big).
\end{eqnarray*}
Note that   $$\lof \dist(x,E)\approx \log{\|(z,w)\|}\approx \min \{-\log|z|,-\log|w|\}\approx m(-\log|z|,-\log|w|   )$$
and $\tilde\psi'_\alpha(\zeta)\approx \|(z,w)\|.$
Combining   these and  \eqref{e:leaf_equation_bis}, the result follows.
\endproof

\begin{lemma}\label{L:theta_eps}
For  every $0<\epsilon\ll 1,$ there is  a smooth function  $ \theta_\epsilon:\  X\to [0,1]$  such that
$\theta_\epsilon(x)=0$ for $\dist(x,E))\leq \epsilon^2$ and $\theta_\epsilon(x)=1$ for $ \dist(x,E))\geq \epsilon,$ and 
\begin{eqnarray*}
|d \theta_\epsilon(x)|_{L_x}|&\leq&  c|\log \epsilon|^{-1}(\dist(x,E))^{-1},\\
|\ddc \theta_\epsilon(x)|_{L_x}|&\leq&  c|\log \epsilon|^{-2}(\dist(x,E))^{-2}
 \end{eqnarray*}
 for  all $x\in X\setminus E.$  Here $c>1$ is  a constant independent of $\epsilon.$ 
\end{lemma}
\proof First we construct  $\theta_\epsilon$  on a singular flow box $\U_a,$ $a\in E.$
Let $\rho:\ \R\to  [0,1]$ be  a  smooth   function   such that  $\rho(t)=1$ for $t\leq 1$  and $\rho(t)=0$ for $t\geq 2.$
For $ x=(z,w)\in \D^2\simeq \U_a,$ we can find  $\zeta=u+iv\in\C$ and $\alpha\in \A$ such that  $x=\tilde\psi_\alpha(\zeta)$
using   \eqref{e:leaf_equation_bis} and \eqref{e:|z|}. Now  we  set
 $$\theta_\epsilon(x):= \rho\Big({m\big(v, (\Im \lambda) u+(\Re \lambda) v\big)\over  |\log\epsilon| }\Big),$$
 where $m$ is  the  the   ``regularized min'' function constructed in Lemma \ref{L:W_bull}.
 For  $x\in  X\setminus  \bigcup_{a\in E}\U_a,$  we  simply set $\theta_\epsilon(x)=1.$
 
 As in Lemma \ref{L:W_bull} we check  that all desired properties are satisfied.
\endproof

\begin{lemma}\label{L:convergence} There is a constant $c>0$ such that
$$
\int_X W_\bullet(x) d\nu_n(x)<c \qquad\text{for}\qquad n\geq 1.
$$
\end{lemma}

\proof By  \eqref{e:nu_n},
$\nu_n$  does not give mass to  $E$ as  $T$ is  a positive $\ddc$-current of bidimension $(1,1).$  
On the  other hand,  by Lemma \ref{L:theta_eps}, $\theta_\epsilon\leq \ind_{X\setminus E}$ and $\theta_\epsilon\to\ind_{X\setminus E}$  on $X$ as $\epsilon\searrow 0,$ and the convergence is local  uniform
on $X\setminus E.$
Therefore, we get that 
$$
\int_X W_\bullet(x) d\nu_n(x)=\lim\limits_{\epsilon\to 0}    \int_X\theta_\epsilon W_\bullet(x) d\nu_n(x).
$$
To  prove the lemma we need to show that 
the   integral in the right hand-side is  
bounded independently of $n$ and $\epsilon.$ By  \eqref{e:nu_n},
this integral is
equal to  
\begin{equation*}
 \langle  \theta_\epsilon W_\bullet, \ddc \psi_n\wedge T\rangle + \langle \theta_\epsilon  W_\bullet, c_1(\Cotan(\Fc),h^*_0)\wedge T\rangle 
 +\epsilon_n   \langle  \theta_\epsilon W_\bullet, g_X \wedge T\rangle.
\end{equation*}
Using the  inequality $0\leq \theta_\epsilon\leq 1$ of  Lemma \ref{L:theta_eps}, and 
using  the  second  inequality  of  Lemma \ref{L:W_bull}, the second term is bounded and the last term tends to $0$  as $n$ tends to infinity.
Hence, it suffices to show that  the first term is  bounded. Since $\theta_\epsilon W_\bullet$  is a smooth function on $X,$
 by Stokes' theorem the first term is  equal to 
\begin{equation*} \langle\ddc  (\theta_\epsilon  W_\bullet)\wedge T,  \psi_n\rangle+ \langle \dbar  (\theta_\epsilon  W_\bullet)\wedge \partial T,  \psi_n\rangle
+  \langle \partial  (\theta_\epsilon  W_\bullet)\wedge \dbar T,  \psi_n\rangle=:I_1+I_2+I_3.
\end{equation*}
Now  we show  that $I_1,$ $I_2$ and $I_3$  are all uniformly bounded independently of $n$ and $\epsilon.$
 
 We expand  $I_1$ and obtain  that
 \begin{eqnarray*}
  I&=&\langle\theta_\epsilon\ddc    W_\bullet\wedge T,  \psi_n\rangle+\langle\partial \theta_\epsilon\wedge \dbar    W_\bullet\wedge T,  \psi_n\rangle
  +\langle\dbar\theta_\epsilon\wedge \partial    W_\bullet\wedge T,  \psi_n\rangle+\langle  W_\bullet\ddc\theta_\epsilon  \wedge T,  \psi_n\rangle\\
  &=:&I_{11}+I_{12}+I_{13}+I_{14}.
 \end{eqnarray*}
Using the  inequality $0\leq \theta_\epsilon\leq 1$ of  Lemma \ref{L:theta_eps}  and
using the fourth inequality of Lemma \ref{L:W_bull}   and  \eqref{e:bound_psi_n},  we  see that
$$
|I_{11}|\lesssim \int  (\dist(x,E))^{-2}(\lof \dist(x,E))^{-2}(\lof \lof \dist(x,E)) g_X(x)|_{L_x}\wedge T(x).
$$
Using this and  \eqref{e:mu} and \eqref{e:relation_Poincare_Hermitian_metrics}, and applying Lemma \ref{L:Poincare}, we infer that $|I_{11}|$  is  bounded by
$
\int_X \lof\lof \dist(x,E) d\mu(x),
$  which is  finite  by \eqref{e:integrability}.
Hence $I_{11}$ is uniformly  bounded independently of $n$ and $\epsilon.$

 Using  the estimate on $\ddc\theta_\epsilon$  of  Lemma \ref{L:theta_eps} and the fact that this form has a support in $\{x\in X:\  \epsilon^2\leq \dist(x,E)\leq \epsilon\}$
 and the estimate on  $   W_\bullet$  of Lemma \ref{L:W_bull}  and  \eqref{e:bound_psi_n}, we deduce that
$$
|I_{14}|\lesssim \int  (\dist(x,E))^{-2}(\lof \dist(x,E))^{-2}(\lof \lof \dist(x,E))^2 g_X(x)|_{L_x}\wedge T(x).
$$
Using this and  \eqref{e:mu} and \eqref{e:relation_Poincare_Hermitian_metrics}, and applying Lemma \ref{L:Poincare}, we infer that $|I_{14}|$  is  bounded by
$
\int_X (\lof\lof \dist(x,E))^2 d\mu(x),
$  which is  finite  by \eqref{e:integrability}.
Hence $I_{14}$ is uniformly  bounded independently of $n$ and $\epsilon.$

Similarly, using  the estimate on $d\theta_\epsilon$ of  Lemma \ref{L:theta_eps}  and the estimate on  $d    W_\bullet$  of Lemma \ref{L:W_bull}  and  \eqref{e:bound_psi_n}, we infer that
 $I_{12}$  and $I_{13} $ are bounded.
We  have shown that $I$ is uniformly  bounded independently of $n$ and $\epsilon.$

 Let  $\tau$
be a $(1,0)$-form 
defined almost everywhere with respect to $\mu=T\wedge g_P$ on $X\setminus $  such that
$\partial T=\tau\wedge T .$  
 In a regular flow box $\U$ as in   \eqref{e:decomposition},  we  see that  $\tau= h_\alpha^{-1}\partial h_\alpha$ on the  plaque passing  through $\alpha\in\Sigma$
for $\nu$-almost every $\alpha$. 
 Using \eqref{e:bound_psi_n}  and   the estimate on $\theta_\epsilon,$  $d\theta_\epsilon$ of  Lemma \ref{L:theta_eps}  and the estimate on $W_\bullet,$  $d    W_\bullet$  of Lemma \ref{L:W_bull},
we see that
\begin{equation*}
 |I_2|\leq \int_X O(\dist(x,E))^{-1} (\lof\dist(x,E))^{-1} (\lof \lof\dist(x,E))^{2} \big)\wedge  \tau \wedge T.
\end{equation*}
Applying  Cauchy-Schwarz inequality,  and then using  \eqref{e:mu}, \eqref{e:relation_Poincare_Hermitian_metrics} 
and Lemma \ref{L:Poincare}, we  infer that
\begin{eqnarray*}
 |I_2|^2&\leq& \big(\int_X (\dist(x,E))^{-2}(\lof \dist(x,E))^{-2}(\lof \lof\dist(x,E))^{2} \cdot T\wedge g_X\big) \\
 &\cdot& \big(\int_X (\lof \lof \dist(x,E))^2 i\tau\wedge \bar \tau  \wedge T)\\
 &\lesssim&\big( \int_X (\lof\lof \dist(x,E))^2d\mu(x)\big) \big(\int_X (\lof \lof \dist(x,E))^2 i\tau\wedge \bar \tau  \wedge T) .
\end{eqnarray*}
By  \cite[ Proposition 3]{FornaessSibony10} (see also the proof of  \cite[Proposition 5.28]{NguyenVietAnh19}), we get the following inequality
 $$i\tau\wedge \overline{\tau}\wedge T\leq  T\wedge g_P=\mu\quad\text{on}\quad X\setminus E.$$
Putting  these  estimates together,   we  get that
\begin{equation*}
 |I_2|\leq  \int_X (\lof \lof \dist(x,E))^2 d\mu(x).
\end{equation*}
So by   \eqref{e:integrability}, $I_2$ is uniformly  bounded independently of $n$ and $\epsilon.$
The same argument also shows that $I_3$ is  uniformly  bounded independently of $n$ and $\epsilon.$
The proof is thereby completed.
\endproof

 Now  we arrive  at the

\proof[Proof of the second identity of  Theorem A]
Since the continuous functions  $\psi_n$ converge  to $\psi$  locally uniformly on $X\setminus E$   as $n\to\infty,$ we deduce from \eqref{e:mu_bis} and \eqref{e:nu_n} that
$\nu_n\to \mu$ on $X\setminus E.$ 

 On the other hand, consider the space  $\Cc^\bullet (X)$  of all  continuous function $f:\  X\setminus E\to\R$  such that  
 \begin{equation*}
  \|f\|_\bullet:= \sup\limits_{x\in X\setminus E}  {|f(x)|\over  W_\bullet(x) } <\infty.
 \end{equation*}
Note that $\Cc^\bullet(X)$  endowed with the norm $ \|\cdot\|_\bullet$ is a Banach space.
By Lemma \ref{L:convergence}, $\nu_n$ are continuous  linear forms on  $\Cc^\bullet (X)$ with  norms $\leq c.$
Hence,   every cluster limit of this  sequence is also a continuous  linear form on  $\Cc^\bullet (X)$ with  norm $\leq c.$
Let $\nu$ be  such a cluster limit.  
We will show   that $\nu=\mu.$ Indeed,
the previous paragraph shows that $\nu=\mu$ on $X\setminus E.$
 Moreover, by \eqref{e:mu}, $\mu$ does not give mass to $E.$
As the function  $  W_\bullet$ is  integrable with respect to $\nu,$
$\nu$ does not give  mass to $E.$
Hence,  $\nu=\mu$ on $X.$
Consequently, $\nu_n\to\nu$ on $\Cc^\bullet (X).$
 Applying this   convergence to  the function $\textbf{1}$ and using \eqref{e:mu_bis} and \eqref{e:nu_n},  we get that
\begin{equation*}
\langle \ddc \psi_n, T\rangle +\langle c_1(\Cotan(\Fc),h^*_0),T\rangle +\epsilon_n \langle g_X, T\rangle\to \langle \textbf{1},\mu\rangle\quad\text{as}\quad n\to\infty.
\end{equation*}
The right-hand side is  $\|\mu\|.$
On the other hand, on the  left-hand side, $\langle c_1(\Cotan(\Fc),h^*_0),T\rangle=c_1(\Cotan(\Fc))\smile \{T\}$  since
$h^*_0$ is a smooth metric.
Moreover, 
$\langle \ddc \psi_n, T\rangle=0$ because  $T$ is $\ddc$-closed, and $\epsilon_n \langle g_X, T\rangle=O(\epsilon_n)\to 0.$
Therefore,  the above limit implies identity \eqref{e:coho_mass}. 
\endproof

\proof[End of  the proof of  Corollary C](see also \cite[Proposition 3.12]{DeroinKleptsyn})
By Theorem  \ref{T:DNS} let $T$ be the unique directed positive harmonic  current such that $\mu$ given by \eqref{e:mu} is a probability measure.
Therefore, it follows from identity \eqref{e:coho_mass} that  $c_1(\Cotan(\Fc))\smile\{T\}=1.$ On the other hand,
it  is well-known (see e.g. \cite{Brunella00}) that $\Nor(\Fc)$ is  equal to $\Oc(d+2)$ and $\Cotan(\Fc)$ is  equal to  $\Oc(d-1).$
Therefore, by  identity \eqref{e:coho_exponent},
\begin{eqnarray*}
\chi(\Fc)=\chi(T)=-\{c_1(\Nor(\Fc))\}\smile \{T\}&=&-(d+2)\{ c_1(\Oc(1))\}\smile \{T\}\\
&=&-{d+2\over d-1}c_1(\Cotan(\Fc))\smile \{T\}\\
&=&-{d+2\over d-1}.
\end{eqnarray*}
Hence, the result follows.
\endproof

 \section{Negative Lyapunov exponent}\label{S:Negativity}
 This  section is  devoted to the proof of  Theorem B.
 By Theorem  \ref{T:Main_1} we may assume without loss of generality that $g_X$ coincides with the Euclidean metric in a local model for every singular flow box $\U_a$ with   $a\in E.$ 
 Consider the space  $\Cc^\star(X)$  of all  continuous functions $f:\  X\setminus E\to\R$  such that  
 \begin{equation*}
  \sup\limits_{x\in X\setminus E}  {|f(x)|\over W (x)} <\infty.
 \end{equation*}
 Consider  the  norm
 \begin{equation*}
  \|f\|_\star:= \sup\limits_{x\in X\setminus E}  {|f(x)|\over W_\star(x)} <\infty,
 \end{equation*}
 where 
 \begin{equation}\label{e:W_star} W_\star(x):= \lof \dist(x,E)\cdot  W (x)\qquad\text{for}\qquad x\in X\setminus E.
 \end{equation}
By  Theorem \ref{T:DNS}, let $T$ be the unique  directed positive  $\ddc$-closed current and let $\mu$ be 
the  measure  associated to $T$ by \eqref{e:mu}. So $\mu$ is  a probability measure.
 Consider the function $\kappa$  defined in \eqref{e:kappa}.
 
 \begin{proposition}\label{P:F_membership}
  \begin{enumerate}
   \item The function $\kappa$  belongs to  $\Cc^\star(X).$
   \item  $\Dc(X\setminus E)$ is  a dense subspace of $\Cc^\star(X).$ 
  \end{enumerate}

 \end{proposition}
\proof \noindent {\bf Proof of assertion (1).} Since the Poincar\'e metric $g_P$ is  leafwise smooth and transversally continuous on $X\setminus E,$ we
deduce  that  $\kappa\in\Cc(X\setminus E).$
Next, we infer from Lemma \ref{L:Poincare} that in a regular flow box $\kappa(x):=\Delta_P\kappa_x(0)$ is  a  bounded function.
Since $W(x)\geq 1,$  it follows that $\kappa(x)\lesssim W(x)$ in this case.

On the other hand,   by Lemma \ref{L:laplacian_hol} (1) we see that for a point $x=(z,w)$ in the local model of  a  singular flow box $\U_a\simeq\D^2,$ $a\in E,$
$$
\kappa(x)\lesssim {|z|^2|w|^2\over (|z|^2+|w|^2)^2}  (\lof {\|(z,w)\|})^2\lesssim W(x).
$$
So in all cases $\kappa(x)\lesssim W(x).$ Hence, $\kappa$ belongs to $\Cc^\star(X).$

\noindent {\bf Proof of assertion (2).}
Let $f$ be a function in $\Cc^\star(X).$
For  $0<\epsilon<\epsilon_0$ consider  the function
$f_\epsilon:=\theta_\epsilon f,$  where  the familly $(\theta_\epsilon)_{0<\epsilon<\epsilon_0}$
is  given in  \eqref{e:theta}.
Using the first two
  properties of \eqref{e:theta} and  the  assumption  $f\in\Cc(X\setminus E)$ with $\sup\limits_{x\in X\setminus E}  {|f(x)|\over W (x)} <\infty,$
  we see that $f_\epsilon\in\Cc_0(X\setminus E)$ and  
  $$
  \| f-f_\epsilon \|_\star \leq  (\log\epsilon)^{-1}\cdot\Big(\sup\limits_{x\in X\setminus E:\ \dist(x,E)\leq \epsilon }  {|f(x)|\over W (x)}\Big) \to 0\qquad\text{as}\qquad \epsilon\to 0.
  $$
  Hence, $\Cc_0(X\setminus E)$ is a dense subspace of  $\Cc^\star(X).$ Since $\Dc_0(X\setminus E)$ is a dense subspace of $\Cc_0(X\setminus E),$ assertion (2) follows.
\endproof

 In the  next key lemma  we use  Hahn-Banach separation theorem  following an idea of Sullivan \cite{Sullivan} and Ghys \cite{Ghys88}. 
In  the  case of transversally  conformal foliations without singularities, the idea has been used in \cite{DeroinKleptsyn}.
 Note  that  by  Proposition \ref{P:Lyapunov_vs_kappa}, $\int_X \kappa(x)d\mu(x)$ is well-defined.
 
 \begin{lemma}\label{L:ddc_approx}
  Suppose  that $\int_X \kappa(x)d\mu(x)\geq 0.$  Then  there exist a  sequence   of  smooth real-valued functions  $\psi_n$  compactly supported in $X\setminus E$
  and   a  sequence of $(\epsilon_n)\subset \R^+$  such that 
$\epsilon_n\searrow 0$ and  that
  $$
     {\kappa(x) -\Delta_P\psi_n(x)\over W_\star(x)}\geq -\epsilon_n\quad \text{for all}\quad  x\in X\setminus E.
  $$
 \end{lemma}
\proof
Consider the space  $\Ic(X)$ of all function  $u:\ X\setminus  E\to\R$ such that there is a smooth real-valued function
 $f$  compactly supported in $X\setminus E$  (i.e.  $f\in\Dc(X\setminus E)$)  such that  $u=\Delta_Pf.$  We  see that $\Ic(X)$ is   a subspace of $\Cc^\star(X).$
 Let $\overline\Ic(X)$ be the closure of $\Ic(X)$ in $\Cc^\star(X).$ 
 Consider the cone $\Cc^+(X)$ of   all functions $f\in\Cc^\star(X)$ such that $f(x)\geq 0$ everywhere.
 Let 
  $\Qc(X)$ be the quotient of  $\Cc^\star(X)$  by  $\overline\Ic(X)$ and $\pi:\  \Cc^\star(X)\to\Qc(X)$ the canonical projection.
  Let $\overline\Cc^+(X)$ be the closure of $\pi(\Cc^+(X))$ in  $\Qc(X).$  By
 Proposition \ref{P:F_membership} (1),
  the  conclusion of the  lemma is  equivalent to  the fact that $\pi(\kappa)\in\Qc(X)$  belongs to  $\overline\Cc^+(X).$
  Suppose  the  contrary in order to get a contradiction. By Hahn-Banach separation  theorem applied to  $\Qc(X),$ there exists a continuous linear
  functional $\hat\nu:\  \Qc(X)\to\R$ such that $\hat\nu\geq 0$ on $\overline\Cc^+(X)$  and  $\hat\nu(\pi(\kappa))<0.$
  So  the continuous linear functionnal  $\nu:=\hat\nu\circ\pi:\  \Cc^\star(X)\to\R $  satisfies that $\nu\geq 0$ on $\Cc^+(X)$  and  $\nu(\kappa)<0$ and $\nu=0$ on  $\Ic(X).$
  Since  we know  by
 Proposition \ref{P:F_membership} (2) that $\Dc(X\setminus  E)$ is  dense in $ \Cc^\star(X),$ it is also  dense in  $\Dc(\Fc),$ and hence  we infer that $\nu$ is  a nonzero  finite positive measure  such that the restriction of $\nu$ to 
  $X\setminus E$ defines a nonzero positive harmonic measure.
  Moreover, $\nu$ does not give mass to the set $E,$ since the function $x\mapsto \dist(x,E)$ belongs to $\Cc^\star(X),$ and hence the Dirac mass $\delta_a$ at any point
  $a\in E$  cannot be evaluated at this function.  Consequently, by Proposition \ref{P:harmonic_currents_vs_measures},
  $\nu=T'\wedge g_P$ for a positive directed $\ddc$-closed current $T'.$
  Therefore, it follows from Theorem \ref{T:DNS} and formula \eqref{e:mu} that $\nu$ is equal to $\mu$  up to a  positive multiplicative constant.  
 Hence,  $\mu(\kappa)$  which  is  of the same sign as  $\nu(\kappa)$  should be  negative, and  we reach  a  contradiction. 
\endproof
\begin{remark}\label{R:ddc_approx}\rm
In Lemma  \ref{L:ddc_approx}
we cannot use the following natural  norm  for $\Cc^\star(X):$ 
\begin{equation*}
  \|f\|:= \sup\limits_{x\in X\setminus E}  {|f(x)|\over W (x)} <\infty.
 \end{equation*}
 Indeed, with this  norm  $\Dc(X\setminus E)$ is  not dense in $\Cc^\star(X):$ the function $x\mapsto \dist(x,E)$ does not belong to the closure of $\Dc(X\setminus E).$

  By adding a constant $c_n$ to  each  $\psi_n,$  we obtain a new  sequence of functions $(\psi_n)$ still satisfying the estimate of Lemma \ref{L:ddc_approx}. Note that
   $\psi_n$ is not necessarily  compactly supported in $X\setminus E.$ In fact,   $\psi_n$ is  constant  near $E.$
\end{remark}

  Consider the  volume form $\Vol$   on $X\setminus E$   given by
  \begin{equation}\label{e:trans_metric}
  \int_{X\setminus E}  fd\Vol=\langle \omega^\perp_X, fg_P\rangle\qquad\text{for}\qquad f\in\Dc(\Fc),
  \end{equation}
where  the transveral form $ \omega^\perp_X$ is given in \eqref{e:tran_form}.
\begin{lemma}\label{L:mu_n}
 \begin{enumerate}  
  \item We have  that  $\langle  W_\star, \Vol\rangle<\infty.$
  \item By adding a constant $c_n$ to  each function  $\psi_n$  given by Lemma \ref{L:ddc_approx} if necessary (see Remark \ref{R:ddc_approx}),  we  may assume  that
  $\langle  W_\star, \mu_n\rangle=1,$  
where $\mu_n$ is  the  volume form on $X\setminus E$  defined by
\begin{equation}\label{e:mu_n}
\mu_n:=\exp{(-2\psi_n)} \Vol.
\end{equation}  
 \end{enumerate}
\end{lemma}
\proof 
  \noindent {\bf Proof of assertion (1).}
  We only need to work on a local model near a singularity $a\in E.$
  For a point $x$ close to $a,$ write $x=(z,w).$ Assume without loss of generality that the fundamental form associated to  $g_X$ is equal to $ idz\wedge d\bar z+idw\wedge d\bar w.$
  By Lemma \ref{L:Poincare},  we have on $L_x$ that
  $g_P(x)\approx {g_X(x)\over  \|x\|^2(\log{\|x\|})^2}.$
  Moreover, by Lemma  \ref{L:trans_metric_near_sing} (1)
there is a constant $c>0$  such that
$
g^\perp_X(x)
\leq c g_X(x).$ We also infer from \eqref{e:W} that $W(x)\leq 2(\log {\|x\|})^2.$
Hence, by \eqref{e:W_star} $W_\star(x)\leq 2(\log {\|x\|})^3.$
  Therefore,  the condition $\langle W_\star,\Vol\rangle<\infty$  will follow if one can show that
  \begin{equation*}
  \int_{x\in\D^2} (\log {\|x\|})^3{  (i dz\wedge d\bar z)\wedge (idw\wedge d\bar w)\over   \|x\|^2(\log{\|x\|})^2}<\infty,
  \end{equation*}
  Using  the  spherical coordinates, we see easily that the above integral is  convergent. The proof of assertion (1) is thereby completed.
  
\noindent {\bf Proof of assertion (2).} It follows  from  formula \eqref{e:mu_n},  assertion (1) and the fact that the functions $\psi_n$ are  bounded.
  \endproof

   Consider the finite open cover  $\Uc=(\U_p)_{p\in I}$ of $X$ given in Subsection \ref{SS:local_model}.    
So   $E\subset I$ and  for every $a\in E,$  $\U_a$ is a singular flow box associated  to $a.$
   For each (regular or singular) flow box 
    $\U_p\in\Uc_p$ with  foliated chart
$\Phi_p:\ \U_p\to \B_p\times \Sigma_p,$ let $\Upsilon_p$  be the volume form on $\Sigma_p$ and  $\varphi_{\U_p}$ be the real-valued function on $\U_p$  given in  \eqref{e:local_varphi},
\eqref{e:Upsilon} and
\eqref{e:local_varphi_bis}.
 For every $n\geq 1$ consider  the function  $\varphi_{n,p}:\ \U_p\to\R$ given by  
 \begin{equation}
 \label{e:varphi_n,p}\varphi_{n,p}:=\psi_n+\varphi_{\U_p} \qquad\text{ on}\qquad \U_p.
 \end{equation}
 We make the   following important observation.
\begin{remark}\label{R:well-defined-bis}\rm By Remark \ref{R:well-defined}, $\Delta \varphi_{\U_p},$ $|d\varphi_{\U_p}|_P,$  $\Delta_P \varphi_{n,p}$ and $|d\varphi_{n,p}|_P$
are independent of $p\in I,$ in other words, they are   well-defined functions 
on the whole $X\setminus E.$ So  we will  write $\Delta_P \varphi_{n}$ (resp.  $|d\varphi_{n,p}|_P$) instead of $\Delta_P \varphi_{n,p}$ (resp.  $|d\varphi_{n,p}|_P$).
\end{remark} 

By  \eqref{e:mu_n}, \eqref{e:trans_metric} and \eqref{e:change-variables} we get
\begin{equation}\label{e:mu_n_bis}
(\Phi_p)_*\mu_n= \exp{(-2\varphi_{n,p}\circ \Phi^{-1}_p)}(\Phi_p)_*(g_P)\wedge\Upsilon_p\qquad\text{on}\qquad \B_p\times\Sigma_p.
\end{equation}

\begin{lemma}\label{L:double_integrability} For every $n\geq 1$ and  every $a\in E,$ in the local model for the singular flow box, 
 \begin{equation*}
  \int_{\U_a} |\Delta_P\varphi_{n,a}|d\mu_n<\infty\quad\text{and}\quad \int_{\U_a} |d\varphi_{n,a}|^2_Pd\mu_n<\infty.
 \end{equation*}
\end{lemma}
\proof

To prove  the first inequality of the lemma, observe that
$$
\int_{\U_a} |\Delta_P\varphi_{n,a}|d\mu_n\leq \int_X |\Delta_P\psi_n|d\mu_n+ \int_{\U_a} |\Delta_P\varphi|d\mu_n.
$$
Since $\psi_n$ is  a smooth function, we have that $|\Delta_P\psi_n|\lesssim  \|\psi_n\|_{\Cc^2}<\infty.$
Hence,  $$\int_X |\Delta_P\psi_n|d\mu_n\lesssim  \langle 1,\mu_n\rangle\lesssim  \langle W_\star,\mu_n\rangle=1.$$ 
By Lemma \ref{L:laplacian_hol}, we have $|(\Delta_P\varphi)(x)|\lesssim W(x)$ for   $x\in\U_a.$
Using this and applying Lemma  \ref{L:mu_n} (2) yield that
$$\int_{\U_a} |\Delta_P\varphi|d\mu_n\lesssim \langle W_\star,\mu_n\rangle=1.$$
The last three  estimates  imply the first inequality of the lemma.

We turn to the proof of  the second  inequality of the lemma.
For $r>0$ let $\D_P(r)$ denote the   disc with center $0$ and radius $r$ with respect to  the Poincar\'e metric $g_P$ on $\D.$
Let $\theta:\ \D_P(3)\to [0,1]$  be a  smooth  function such that  $\theta=1$ on    $\D_P(1)$ and $\theta=0$ outside     $\D_P(2)$ and
$|\Delta_P\theta|<100.$
Let  $\chi:\ \D_P(2)\to\R$ be smooth function.  
  Applying Stokes' theorem on $\D_P(2)$ yields that
  \begin{equation*}
 \int_{\D_P(2)} \ddc\theta  \exp{(-2\chi)}=\int_{\D_P(2) } \theta \ddc \exp{(-2\chi)}.
\end{equation*}
Using  \eqref{e:Delta_P} and \eqref{e:df_P}  we deduce that 
\begin{equation*}
 \int_{\D_P(2)} \theta |d\chi|^2_P \exp{(-2\chi)}g_P=\int_{\D_P(2) } (\Delta_P\theta )\exp{(-2\chi)}g_P+\int_{\D_P(2)} \theta (\Delta_P\chi) \exp{(-2\chi)}g_P.
\end{equation*}
  Hence,  
  \begin{equation*}
 \int_{\D_P(1)}  |d\chi|^2_P \exp{(-2\chi)}g_P\leq 100\int_{\D_P(2) }  \exp{(-2\chi)}g_P+\int_{\D_P(2)} | \Delta_P\chi| \exp{(-2\chi)}g_P.
\end{equation*}
Fix $\alpha\in \A.$ The above  inequality  implies that for every $x\in\mathcal L_a,$
 \begin{equation*}
 \int_{\pi_x(\D_P(1))}  |d\varphi_{n,a}|^2_P \exp{(-2\varphi_{n,a})}g_P\leq 100\int_{\phi_x(\D_P(2)) }  \exp{(-2\varphi_{n,a})}g_P+\int_{\phi_x(\D_P(2))} | \Delta_P\varphi_{n,a}| \exp{(-2\varphi_{n,a})}g_P,
\end{equation*}
where $\phi_x$ is given in \eqref{e:covering_map}.
Integrating  both sides of the  above inequality with respect to $g_P(x),$ $x\in \mathcal L_a,$  we see that there is a universal constant $c>0$ independent of $\alpha\in\A$ and $n\in\N$ such that
\begin{equation*}
 \int_{\mathcal L_\alpha}  |d\varphi_{n,a}|^2_P \exp{(-2\varphi_{n,a})}g_P\leq c\int_{\mathcal L_\alpha}  \exp{(-2\varphi_{n,a})}g_P+c\int_{\mathcal L_\alpha} |\Delta_P\varphi_{n,a} |\exp{(-2\varphi_{n,a})}g_P.
\end{equation*}
Integrating both sides of the last inequality with respect to the volume form $\Upsilon_a$ on $\A$ and using \eqref{e:change-variables} and \eqref{e:trans_metric}, \eqref{e:mu_n} and \eqref{e:varphi_n,p},
  we obtain that
 \begin{equation*}
 \int_{\U_a} |d\varphi_{n,a}|^2_P d\mu_n\leq c \int_{\U_a}  d\mu_n+c\int_{\U_a} |\Delta_P\varphi_{n,a}| d\mu_n.
\end{equation*}
On the other hand, by Lemma \ref{L:mu_n} (2), $\langle \textbf{1},\mu_n\rangle\leq \langle W_\star,\mu_n\rangle =1.$
Hence, the second  inequality of the lemma follows from   the first one.
\endproof

  \begin{lemma}\label{L:Stokes}
  For every $n\geq 1$ we have that
  $$ \int_X (\Delta_P \varphi_n-|d\varphi_n|^2_P) d\mu_n=0.$$
  \end{lemma}
   By Lemma  \ref{L:double_integrability}, the above  integral makes sense. 
\proof 
Fix 
 a partition of  unity $(f_p)_{p\in I}$ associated to  $\Uc:$ $\sum_{p\in I} f_p= \mathbf{1},$  where  the support of  each function $f_p$ is  contained in  $\U_p.$ 
 Using  \eqref{e:Delta_P} and \eqref{e:change-variables} and then 
applying  Stokes' theorem to each  plaque    of   each regular flow box  $\U_p,$ we  get that
\begin{eqnarray*}
{1\over \pi}\int_X \Delta_Pf_p d\mu_n&=&\int_{\B_p\times \Sigma_p} \ddc (f_p\circ \Phi^{-1}_p) \exp{(-2\varphi_{n,p}\circ \Phi^{-1}_p)}  \Upsilon_p\\
&=& \int_{\Sigma_p} \big(\int_{\B_p} \ddc (f_p\circ \Phi^{-1}_p) \exp{(-2\varphi_{n,p}\circ \Phi^{-1}_p)}\big)  \Upsilon_p\\
&=& \int_{\Sigma_p} \big(\int_{\B_p}(f_p\circ \Phi^{-1}_p) \ddc\exp{(-2\varphi_{n,p}\circ \Phi^{-1}_p)} \big) \Upsilon_p.
\end{eqnarray*}
A straightforward computation shows that the expression in the last line is equal to 
\begin{eqnarray*}
 &&\int_{\Sigma_p} \big(\int_{\B_p}(f_p\circ \Phi^{-1}_p)(-2 \ddc(\varphi_{n,p} \circ \Phi^{-1}_p)+4i\partial (\varphi_{n,p}\circ \Phi^{-1}_p)\wedge\dbar(\varphi_{n,p}\circ \Phi^{-1}_p))
 \exp{(-2\varphi_{n,p}\circ \Phi^{-1}_p)} \big) \Upsilon_p\\
&=&  \int_{\Sigma_p} \big(\int_{\B_p}(f_p \circ \Phi^{-1}_p)(-\Delta_P(\varphi_{n,p}\circ \Phi^{-1}_p)+ |d\varphi_{n,p}\circ \Phi^{-1}_p|^2_P) \exp{(-2\varphi_{n,p}\circ \Phi^{-1}_p)}g_P\big)\wedge \Upsilon_p\\
&=&\int_{X} f_p (-\Delta_P\varphi_{n,p}+ |d\varphi_{n,p}|^2_P)d\mu_n,
\end{eqnarray*}
where the first  equality holds  by \eqref{e:Delta_P}-\eqref{e:df_P}, and the  second one  by \eqref{e:change-variables}.
We have shown that for $p\in I\setminus E,$
\begin{equation}\label{e:Stokes_outside_sing}
 \int_X \Delta_Pf_p d\mu_n=\int_{X} f_p (-\Delta_P\varphi_{n,p}+ |d\varphi_{n,p}|^2_P)d\mu_n.
\end{equation}
We will show that for each $a\in E,$
\begin{equation}\label{e:Stokes_sing}
 \int_X \Delta_Pf_a d\mu_n=\int_{X} f_a (-\Delta_P\varphi_{n,a}+ |d\varphi_{n,a}|^2_P)d\mu_n.
\end{equation}
Taking \eqref{e:Stokes_sing} for granted, we will complete the proof of the lemma. Indeed,
by summing up \eqref{e:Stokes_outside_sing} and  \eqref{e:Stokes_sing}, and using that $\Delta_P\textbf{1}=0,$  we infer that
\begin{eqnarray*}
 0&=&\int_X \Delta_P \mathbf{1} d\mu_n=\sum_{p\in I}\int_X \Delta_P f_pd\mu_n\\
 &=&\sum_{p\in I\setminus E}\int_X f_p (-\Delta_P\varphi_{n,p}+ |d\varphi_{n,p}|^2_P)d\mu_n+
 \sum_{a\in E}\int_X  f_a (-\Delta_P\varphi_{n,a}+ |d\varphi_{n,a}|^2_P)d\mu_n\\
 &=&\sum_{p\in I} \int_X f_p(-\Delta_P\varphi_{n}+ |d\varphi_{n}|^2_P)d\mu_n\\
 &=&\int_X ( \sum_{p\in I}  f_p)(-\Delta_P\varphi_{n}+ |d\varphi_{n}|^2_P)d\mu_n.
\end{eqnarray*}
 Hence, the desired identity of the lemma follows. 
 
 To complete the proof it remains to establish \eqref{e:Stokes_sing}.
 We  work  in the local model $\Fc|_{\U_a}\simeq (\D^2,\Lc,\{0\})$  associated  to  a point $a\in E.$
 Using \eqref{e:local_varphi_bis} and \eqref{e:varphi_n,p},
 we obtain   a form $\Upsilon_a$ which is well-defined on the distinguished transversal $\A_a$ of this  model
 as well as  function $\varphi_{n,a}$ which is  well-defined on $\D^2.$

Let  $\rho:\ \C\cup\{\infty\} \to \R^+$ be a smooth function 
which  satisfies
\begin{equation*}
\rho(|t|)= 0\ \text{for}\ |t|\geq 2,\quad
\rho(t)= 1\ \text{for}\ |t|\leq 1/2,\quad  0<\rho(t)< 1\ \text{for}\ 1/2<| t|< 2.
\end{equation*}
 For every $0<\epsilon\ll 1,$ consider the smooth function $\rho_\epsilon:\ \D^2\to [0,1]$ given by
 \begin{equation}\label{e:rho_ep}
 \rho_\epsilon(x):=\big(1-\rho(2z/\epsilon)\big)\big(1-\rho(2w/\epsilon)\big) \quad\text{for}\quad x=(z,w)\in\D^2.
 \end{equation} 
 Note that  $\rho_\epsilon(x)=1$  if $|z|> \epsilon$ and $|w|>\epsilon,$ 
 $\rho_\epsilon(x)=0$ if either  $|z|<\epsilon/4$  or  $|w|<\epsilon/4.$ 
 Moreover,  
 we have 
 \begin{equation}\label{e:deriv_rho_ep}
 \partial_z \rho_\epsilon=O(\epsilon^{-1})dz,\quad \dbar_z\rho_\epsilon=O(\epsilon^{-1})d\bar z,\quad\ddcz
 \rho_\epsilon=O(\epsilon^{-2}) idz\wedge d\bar z,\quad \partial_z\dbar_w
 \rho_\epsilon=O(\epsilon^{-2}) idz\wedge d\bar w.
 \end{equation}
 Similar  estimates also hold when  $z$ is replaced by $w.$ 
 
 Now fix $\alpha\in\A$  and consider  the function $( \Delta_P f_a)  \rho_\epsilon\exp{(-2\varphi_{n,a})}$ restricted to the  Riemann surface $\mathcal L_\alpha.$
 Recall  that   $\mathcal L_\alpha$ is  the image of the sector $\S$  by 
  the  map $\tilde\psi_\alpha$ given in \eqref{e:leaf_equation_bis} and that $(z,w)=\tilde\psi_\alpha(\zeta)$ is related to $\zeta=u+iv$ by  \eqref{e:|z|}. 
 Since $f_a=1$ in a neighborhood of $a=0$ and $f_a$ is  compactly supported in $\D^2,$ it follows from the above properties of $\rho_\epsilon$ that
 the support of the function $( \Delta_P f_a)  \rho_\epsilon\exp{(-2\varphi_{n,a})}$   is  contained in the image by $\tilde\psi_\alpha$ of a  set $K_\alpha$
 such that
 $$   K_\alpha \Subset\{(u,v)\in\S:  0< v<-c\log{(\epsilon/4)}\quad\text{and}\quad 0< (\Im\lambda)u+(\Re\lambda)v<-c\log{(\epsilon/4)} \} \subset \S.$$
 Here $c>0$ is a constant that depends only on  $a.$ 
 Therefore, we are able to 
 applying Stokes' theorem  to this  function  on   $\mathcal L_\alpha.$ Consequently, for each $\alpha\in\A$ we have that 
 \begin{equation}\label{e:Stokes_cut_off}
 \int_{\mathcal L_\alpha}( \Delta_P f_a)  \rho_\epsilon\exp{(-2\varphi_{n,a})} g_P= \int_{\mathcal L_\alpha} f_a   \Delta_P\big(\rho_\epsilon\exp{(-2\varphi_{n,a})}\big) g_P.
\end{equation}
 Since  $\rho_\epsilon\nearrow 1$ on $\D^2$ as $\epsilon\to 0,$ the left-hand side of \eqref{e:Stokes_cut_off} tends to $\int_{\mathcal L_\alpha}( \Delta_P f_a)  \exp{(-2\varphi_{n,a})} g_P.$
 On the other hand, 
as the  difference $\Delta_P\big(\rho_\epsilon\exp{(-2\varphi_{n,a})}\big)- \rho_\epsilon\Delta_P\big(\exp{(-2\varphi_{n,a})}\big)$ is  equal to $0$  on  every open set where $\rho_\epsilon$ is  
identically either  to $0$ or $1,$
we deduce from the above  properties of $\rho_\epsilon$  that the above difference 
is  nonzero only if $x\in \mathcal D_\epsilon,$ where
$$\mathcal D_{\epsilon}:=\left\lbrace x=(z,w)\in\D^2:\ \epsilon/4\leq |z|,|w|\leq \epsilon\right\rbrace .$$
Therefore, we infer from  \eqref{e:deriv_rho_ep}  that for $x=(z,w)\in \mathcal D_{\epsilon}\cap  \mathcal L_\alpha,$ 
 \begin{equation*}
    \Delta_P\big(\rho_\epsilon\exp{(-2\varphi_{n,a})}\big)- \rho_\epsilon \Delta_P\big(\exp{(-2\varphi_{n,a})}\big)=O(\epsilon^{-2}) \chac_{\mathcal D_\epsilon} \exp{(-2\varphi_{n,a})},
 \end{equation*}
 where  $\chac_{\mathcal D_\epsilon}$ is the characteristic function of $\mathcal D_\epsilon.$
 Integrating  both sides with respect to the form  $\Upsilon_a$ on $\A_a$ and using \eqref{e:mu_n}, \eqref{e:mu_n_bis}, we get that
 \begin{equation*}
  \Big| \int_{\alpha\in \A_a}\big(\int_{\mathcal L_\alpha} f_a  \big( \Delta_P\big(\rho_\epsilon\exp{(-2\varphi_{n,a})}\big)- 
  \rho_\epsilon\Delta_P\big(\exp{(-2\varphi_{n,a})}\big)\big) g_P\big) \Upsilon_a(\alpha)\Big|
  \leq c
  \langle\epsilon^{-2}\chac_{\mathcal D_\epsilon}\mu_n\rangle.
 \end{equation*}
 Arguing as in the proof of Lemma \ref{L:mu_n}, we  see easily that  the  right-hand side is bounded by
 \begin{equation*}
  \int_{x=(z,w)\in\mathcal D_\epsilon} {  (i dz\wedge d\bar z)\wedge (idw\wedge d\bar w)\over \epsilon^2  \|x\|^2(\log{\|x\|})^2}=O(|\log \epsilon|^{-2}).
  \end{equation*} 
   Therefore, we get that  as $\epsilon$ tends to $0,$ the difference
     \begin{equation*}
   \int_{\alpha\in \A_a}\big( \int_{\mathcal L_\alpha}f_a   \Delta_P\big(\rho_\epsilon\exp{(-2\varphi_{n,a})}\big) g_P \big)\Upsilon_a(\alpha)
 -\int_{\alpha\in \A_a}\big( \int_{\mathcal L_\alpha} f_a  \rho_\epsilon \Delta_P\big(\exp{(-2\varphi_{n,a})}\big) g_P\big))\Upsilon_a(\alpha)
\end{equation*}
tends also to $0.$
   Letting $\epsilon$ tend to $0,$ we infer 
   from this and \eqref{e:Stokes_cut_off} that  
  \begin{equation*}
  \int_{\alpha\in \A_a}\big( \int_{\mathcal L_\alpha}( \Delta_P f_a)  \exp{(-2\varphi_{n,a})} g_P\big) \Upsilon_a(\alpha) =   \int_{\alpha\in \A_a}\big(\int_{\mathcal L_\alpha} f_a 
  \Delta_P\big(\exp{(-2\varphi_{n,a})}\big) g_P\big) \Upsilon_a(\alpha).
\end{equation*} 
Hence, we get \eqref{e:Stokes_sing} easily  using \eqref{e:mu_n_bis}.
\endproof

\proof[End of the proof of  Theorem B]
 Suppose in order to get  a  contradiction   that $\chi(T)\geq 0.$
 So by Proposition  \ref{P:Lyapunov_vs_kappa}, we have that $$\int_X \kappa(x)d\mu(x)\geq 0.$$ 
By Lemma \ref{L:mu_n} (2), there is a  subsequence  of $(\mu_n)$  converging  weakly to  a  measure $\mu'.$
Consider the family of currents $T_n$ on $\Dc^{1,1}(\Fc)$ defined by
$$
T_n(h):=\big\langle \exp{(-2\psi_n)}  \Vol,h   \big\rangle\qquad\text{for}\qquad  h\in\Dc^{1,1}(\Fc),
$$
where $\Vol$ and $\psi_n$ are defined  in \eqref{e:trans_metric} and \eqref{e:mu_n}.
Let $f_h\in\Dc^0(\Fc)$  be defined by $h= f_h g_P.$ Then we get that
$T_n(h)= \langle f_h,\mu_n\rangle.$
On the other hand, by   Lemma \ref{L:mu_n} (2),  $\langle W_\star,\mu'\rangle=1.$
  So $\mu'$ does not charge any  singular point.   
Consequently,  restricting  $\mu'$ to  each flow box  we can prove that $T_n$  converge weakly to  a nonzero directed positive current $T'$ and $\mu'=T'\wedge g_P$ on $X\setminus E.$

Next, we will show that $T'$ is closed on $X\setminus E.$
Taking  this for granted, 
 we see that $T'$  is a nonzero  directed  positive closed current. Moreover, since
 $$
 \langle T',g_X\rangle \leq \langle T'\wedge g_P,  \chac_X\rangle=\langle \mu', \chac_X\rangle\lesssim  \langle \mu', W_\star\rangle= 1, 
 $$
 by a classical theorem of Skoda, $T'$ extends trivially through $E$ to a positive closed current on $X.$ 
 This contradicts the hypothesis of the theorem, and  
the proof is  thereby completed.

It remains to prove that $dT'=0$ on $X\setminus E.$ To do this
consider  $\xi\in\Dc^{1}(\Fc).$   Using a partition of unity, we can write $\xi$
as a  finite sum of $1$-form  whose  support is contained  a regular  flow box.  Therefore, we   may assume that  the support of $\xi$  is  contained
in a regular  flow box $\U$  with foliated chart    $\Phi:\ \U\to \B\times \Sigma.$ 
Let $\Upsilon$  be the volume form on $\Sigma$ and  $\varphi_{\U}$ be the real-valued function on $\U$  given in 
\eqref{e:Upsilon} and  \eqref{e:local_varphi}.
 For every $n\geq 1,$ following  \eqref{e:varphi_n,p}  we consider  the function  $\varphi_{n,\U}:\ \U\to\R$ given by  $\varphi_{n,\U}:=\psi_n+\varphi_{\U} $ on $\U.$ 
 The  measures $\mu_n$ given in \eqref{e:mu_n}  can be  rewritten as
$$
\mu_n=g_{X,n}^\perp\wedge g_P,
$$
where using  \eqref{e:mu_n_bis}, $g_{X,n}^\perp$ is a  directed $(1,1)$-current on $\Fc$  which can be described in the flow box $\B\times \Sigma\simeq \U$ by 
$$
(\Phi_*g_{X,n}^\perp)(\zeta,\alpha)=\exp{(-2(\varphi_{n,\U}\circ \Phi^{-1})(\zeta,\alpha))}\Upsilon(\alpha)\quad\text{for}\quad (\zeta,\alpha)\in \B\times \Sigma.
$$  
Applying Stokes' theorem on $\B\times \Sigma$ and on  each plaque of $\B\times\{\alpha\}$ with $\alpha\in\Sigma$ and using \eqref{e:change-variables},  we  get that
\begin{eqnarray*}
 \langle d\xi, g_{X,n}^\perp\rangle&=&\int_{\B\times \Sigma} \Phi_*(d\xi)\wedge \exp{(-2\varphi_n\circ \Phi^{-1})}\Upsilon\\
 &=&\int_{\B\times \Sigma}\Phi_*\xi\wedge d\big(\exp{(-2\varphi_n\circ \Phi^{-1})}\Upsilon\big)\\
 &=& \int_\Sigma\big(\int_{\B}\Phi_*\xi\wedge d\big(\exp{(-2\varphi_n\circ \Phi^{-1})}\big)\Upsilon\\
 &=& -  2\int_\Sigma\big(\int_{\B}\Phi_*\xi\wedge (d\varphi_n \circ\Phi^{-1})\cdot \big(\exp{(-2\varphi_n\circ \Phi^{-1})}\big)\Upsilon \\
 &=&-2\int_X (\xi\wedge d\varphi_n)\wedge g_{X,n}^\perp.
\end{eqnarray*}
Hence, by Cauchy--Schwarz's inequality 
\begin{equation}\label{e:Cauchy--Schwarz}
\big|    \langle d\xi, g_{X,n}^\perp      \rangle                  \big|
= 2\big| \int_X (\xi\wedge d\varphi_n)\wedge g_{X,n}^\perp\big|  \leq c \|\xi\|_{\Cc^0}\int_X  |d \varphi_n|_Pd\mu_n
\leq  c \|\xi\|_{\Cc^0}\big(\int_X  |d \varphi_n|^2_Pd\mu_n\big)^{1/2},
\end{equation}
where $c>0$ is a constant depending only on $\U,$ and $ \|\xi\|_{\Cc^0}$ is the  sup-norm of $\xi$  (see Subsection \ref{SS:background}). Here  we recall from  Lemma  \ref{L:Poincare} (1)
that on $\U$ the leafwise Poincar\'e metric $g_P$ is  equivalent to the restriction of the ambient metric $g_X$ on  plaques of $\U.$
Moreover, by Remark \ref{R:well-defined-bis}, $|d\varphi_{n,\U}|_P,$ $\Delta_P\varphi_{n,\U},$ $ \Delta\varphi_\U$  are independent of $U,$ and are   well-defined 
on the  whole $X\setminus E.$ So we will omit the index $\U$ and   rewrite them simply as  $|d\varphi_{n}|_P,$ $\Delta_P\varphi_{n},$ $ \Delta\varphi$ respectively.

By Lemma \ref{L:Stokes},
\begin{equation}\label{e:L_closed_Stokes}
\int_X  |d \varphi_n^2|_Pd\mu_n=\int_X \Delta_P\varphi_nd\mu_n.
\end{equation}
On the other hand, recall  from  \eqref{e:curvature-varphi}  and  Remark \ref{R:well-defined} that  
$\kappa(x)=-\Delta_P\varphi_\U(x)=-\Delta_P \varphi(x)$ for $x\in X\setminus E.$
Therefore, by Lemma \ref{L:ddc_approx}, there is  a  sequence of $(\epsilon_n)\subset \R^+$  such that 
$\epsilon_n\searrow 0$ and  
 $$
      \big({-\Delta_P\varphi(x)- \Delta_P\psi_n(x)\over W_\star(x)}\big)\geq -\epsilon_n\quad\text{for}\quad  x\in X\setminus E.
  $$
  So
  \begin{equation*}
  \int_X \Delta_P(-\varphi_n)d\mu_n= \int_X W_\star(x){\Delta_P(-\varphi-\psi_n)  \over  W_\star(x)}d\mu_n 
                             \geq-   \int_X \epsilon_n W_\star(x)d\mu_n .
  \end{equation*}
  Applying    Lemma \ref{L:mu_n} (2) to the last line yields that 
  $$\limsup_{n\to\infty}\int_X \Delta_P\varphi_nd\mu_n\leq 0.$$
  This,  combined with \eqref{e:L_closed_Stokes}, implies that 
  \begin{equation*} 
\limsup_{n\to\infty}\int_X  |d \varphi_n^2|_Pd\mu_n\leq 0.
\end{equation*}
Hence, $\lim_{n\to\infty}\int_X  |d \varphi_n^2|_Pd\mu_n=0.$
Inserting this  into the right hand side of \eqref{e:Cauchy--Schwarz}, we get that $\lim_{n\to\infty}\langle d\xi, g_{X,n}^\perp\rangle=0.$
Hence, $\langle  dT',\xi\rangle=0.$ Thus, $T'$ is  closed  on $X\setminus E$ as  claimed.

\endproof

\begin{remark} \rm It is  of interest  to  know  whether  Theorems  A and B still hold if $X$ is merely a compact K\"ahler surface. This is  the case
 if  we can relax  the projectivity assumption in 
 \cite[Theorem 1.1]{NguyenVietAnh18b}. 
 The reader may find in  \cite{DinhSibony18b,FornaessSibony08,NguyenVietAnh18c,NguyenVietAnh19}  some other open questions in the ergodic theory of singular holomorphic foliations.
\end{remark}


\medskip

\noindent
Vi{\^e}t-Anh Nguy{\^e}n,  
Universit\'e de Lille, 
Laboratoire de math\'ematiques Paul Painlev\'e, 
CNRS U.M.R. 8524,  
59655 Villeneuve d'Ascq Cedex, 
France.\\
{\tt Viet-Anh.Nguyen@univ-lille.fr},
{\tt http://math.univ-lille1.fr/$\sim$vnguyen}

\end{document}